\documentclass{birkmult}

\usepackage{amsmath,amssymb,amscd,bbm}
\usepackage[mathscr]{eucal}
\usepackage[all]{xy}
\usepackage[utf8]{inputenc}
\usepackage{hyperref}

 \newtheorem{thm}{Theorem}[section]
 \newtheorem{cor}[thm]{Corollary}
 \newtheorem{lem}[thm]{Lemma}
  
 \newtheorem{prop}[thm]{Proposition}
 \theoremstyle{definition}

 \newtheorem{rem}[thm]{Remark}
 \newtheorem{rems}[thm]{Remarks}
 \newtheorem{exa}[thm]{Example}
 
 \numberwithin{equation}{section}

\newcommand{\uni}{\mathrm{uni}}

\newcommand{\spec}{\sigma}

\newcommand{\sct}{\mathrm{se}}

\newcommand{\op}{\mathrm{op}}

\newcommand{\ancgen}[1]{\langle#1\rangle}

\newcommand{\bp}{\mathrm{bp}}

\def\cent{\calZ}

\def\BMb{\calM_\mathrm{b}}

\newcommand{\vanish}[1]{\relax}

\newcommand{\beq}{\begin{equation}}
\newcommand{\eeq}{\end{equation}}
\newcommand{\defiff}{\stackrel{\text{\rm def}}{\Longleftrightarrow}}
\newcommand{\Ball}{\mathrm{Ball}}

\newcommand{\hs}{\hskip-0.1em}
\newcommand{\set}[1]{\hs\left[\,#1\,\right]}

\newcommand{\bfz}{\mathbf{z}}

\newcommand{\Clo}{\mathcal{C}}
\newcommand{\prfnoi}{\smallskip\noindent}

\newcommand{\upi}{\pi}

\newcommand{\emdf}{\bf}

\newcommand{\Id}{\mathrm{I}}

\newcommand{\suchthat}{\,\,|\,\,}

\DeclareMathOperator{\calL}{\mathcal{L}}
\DeclareMathOperator{\calN}{\mathcal{N}}
\newcommand{\calM}{\mathcal{M}}

\DeclareMathOperator{\calA}{\mathcal{A}}
\DeclareMathOperator{\calB}{\mathcal{B}}
\DeclareMathOperator{\calC}{\mathcal{C}}
\DeclareMathOperator{\calD}{\mathcal{D}}
\newcommand{\calE}{\mathcal{E}}
\DeclareMathOperator{\calF}{\mathcal{F}}
\DeclareMathOperator{\calG}{\mathcal{G}}

\DeclareMathOperator{\calT}{\mathcal{T}}

\DeclareMathOperator{\calS}{\mathcal{S}}

\DeclareMathOperator{\calZ}{\mathcal{Z}}

\newcommand{\N}{\mathbb{N}}

\newcommand{\R}{\mathbb{R}}
\newcommand{\C}{\mathbb{C}}

\newcommand{\K}{\mathbb{K}}

\newcommand{\ud}{\mathrm{d}}
\newcommand{\ue}{\mathrm{e}}
\newcommand{\ui}{\mathrm{i}}

\newcommand{\eM}{\mathrm{M}}

\newcommand{\Ha}{\mathrm{H}}

\newcommand{\vphi}{\varphi}

\newcommand{\sector}[1]{\mathrm{S}_{#1}}

\newcommand{\Lap}{\mathcal{L}}

\newcommand{\res}[1]{|_{#1}}

\newcommand{\Sum}[2][\relax]{%
 \ifx#1\relax \sideset{}{_{#2}}\sum 
 \else \sideset{}{^{#1}_{#2}}\sum
 \fi}

\newcommand{\car}{\mathbf{1}}
\DeclareMathOperator{\re}{Re}

\newcommand{\konj}[1]{\overline{#1}}

\newcommand{\abs}[1]{\vert #1 \vert}

\DeclareMathOperator{\Mer}{\mathcal{M}}

\newcommand{\Ce}{\mathrm{C}}

\newcommand{\Cb}{\mathrm{C}_{\mathrm{b}}}

\newcommand{\Ell}[1]{\mathrm{L}_{#1}}

\newcommand{\Meas}{\mathcal{M}}

\newcommand{\resolv}{\varrho}

\DeclareMathOperator{\reg}{reg}
\DeclareMathOperator{\bdd}{bdd}

\newcommand{\ohne}{\setminus}

\newcommand{\leer}{\emptyset}

\newcommand{\dann}{\Rightarrow}

\newcommand{\gdw}{\Leftrightarrow}

\newcommand{\Gdw}{\,\,\Longleftrightarrow\,\,}

\newcommand{\nach}{\circ}

\DeclareMathOperator{\dom}{dom}
\DeclareMathOperator{\ran}{ran}

\newcommand{\cls}[1]{\overline{#1}}

\newcommand{\rand}{\partial}

\newcommand{\spann}{\mathrm{span}}

\newcommand{\BL}{\mathcal{L}}

\newcommand{\norm}[2][\relax]{%
   \ifx#1\relax \ensuremath{\lVert#2\rVert}
   \else \ensuremath{\left\Vert#2\right\Vert_{#1}}
   \fi}

\makeatletter
\newcommand{\sprod}[2]{\ensuremath{%
  \setbox0=\hbox{\ensuremath{#2}}
  \dimen@\ht0
  \advance\dimen@ by \dp0
  \left[ #1\rule[-\dp0]{0pt}{\dimen@}, #2\hspace{1pt}\right]}}
\newcommand{\bsprod}[2]{\ensuremath{%
  \setbox0=\hbox{\ensuremath{#2}}
  \dimen@\ht0
  \advance\dimen@ by \dp0
  \bigl[ #1\rule[-\dp0]{0pt}{\dimen@}, #2\hspace{1pt}\bigr]}}

 \makeatother
\newcommand{\dprod}[2]{\ensuremath{\langle#1,#2\rangle}}

\newcommand{\fourier}[1]{\widehat{#1}}

\newcounter{aufzi}
\newenvironment{aufzi}{\begin{list}{ {\upshape\alph{aufzi})}}{
        \usecounter{aufzi}
        \topsep1ex
        \parsep0cm
        \itemsep1ex
        \leftmargin0.8cm
        \labelwidth0.5cm
        \labelsep0.3cm
}}
{\end{list}}

\newcounter{aufzii}
\newenvironment{aufzii}{\begin{list}{\hfill {\upshape 
(\roman{aufzii})}}{
        \usecounter{aufzii}
        \topsep1ex
        \parsep0cm
        \itemsep1ex
        \leftmargin0.8cm
        \labelwidth0.5cm
        \labelsep0.3cm
         \itemindent0cm
}}
{\end{list}}

\newcounter{aufziii}
\newenvironment{aufziii}{\begin{list}{ {\upshape\arabic{aufziii})}}{
        \usecounter{aufziii}
        \topsep1ex
        \parsep0cm
        \itemsep1ex
        \leftmargin0.8cm
        \labelwidth0.5cm
        \labelsep0.3cm
}}
{\end{list}}

\newcommand{\cnt}{\mathrm{Z}}

\usepackage{hyperref}

\begin{document}
%
%
%
%
%
%
%
%
%
\title{On the Fundamental Principles of Unbounded Functional Calculi}



\author[Markus Haase]{Markus Haase}

\address{%
Kiel University\\
Mathematisches Seminar\\
Ludewig-Meyn-Str.4\\
42118 Kiel, Germany}

\email{haase@math.uni-kiel.de}




\date{\today}
\dedicatory{Dedicated to Wolfgang Arendt and Lutz Weis on the occasion
of their 70th birthdays}

\begin{abstract}
In this paper, a new axiomatization for unbounded functional
calculi is proposed and the associated theory is elaborated,
comprising, among others,  uniqueness and compatibility results and
extension theorems of algebraic and topological nature. 
 In contrast to earlier approaches, no commutativity assumptions
need to be made about the underlying algebras. 

In a second part, the abstract theory is illustrated in  familiar
situations (sectorial operators, semigroup generators). 
New topological extension theorems are proved for
the sectorial calculus and the Hille--Phillips calculus. 
Moreover, it is shown that the Stieltjes and the Hirsch calculus
for sectorial operators are subcalculi of a (small) topological
extension of the sectorial calculus. 
\end{abstract}

\maketitle

\section{Introduction}\label{s.intro}

Some instances of functional calculi are as old as modern mathematics.
The very concept of a functional calculus, however, even nowadays 
still remains partly heuristic, with no definitive and widely accepted 
precise mathematical definition. One reason for this situation
lies in the variety of instances such a definition has to cover. In
particular those calculi are of interest, where unbounded operators are not just
the starting point but the target. 
These calculi are called {\em unbounded} in the following, in
order to distinguish them from the {\em bounded} calculi,  which
by (our) definition are those that yield only bounded operators.
(In particular, we do not intend any continuity or suppose any
topology when we speak of a bounded calculus here.)

To illustrate our terminology, consider the
Borel calculus of a normal operator $A$ on a Hilbert space. 
Even if $A$ is unbounded, the calculus $f\mapsto f(A)$ 
yields bounded operators as long as $f$ is a bounded function. Hence, 
restricting the Borel calculus to bounded functions yields
a bounded calculus in our terminology. In contrast to that,
even if $A$ is bounded, there are Borel functions $f$ such that 
$f(A)$ is unbounded. Hence, the full Borel calculus is an unbounded
functional calculus.

Now, there is little controversy about what a bounded functional
calculus should be, namely  an algebra homomorphism
$\Phi: \calF \to \BL(X)$ where $\calF$ is an algebra and $X$ is a
Banach space. Of course, there are obvious variations possible
concerning, e.g.,  the properties of $\calF$ (unital? commutative? a function algebra?)
or the space $X$ (just locally convex?) or topological requirements
for $\calF$ and $\Phi$. All these, however,  are somehow
unproblematic because the whole situation lies still within the terminological
realm of classical representation theory.
In contrast, unbounded functional calculi lie beyond that realm, as
the set of unbounded operators on a Banach space is not an algebra any
more. There is simply no classical terminological framework to cover
unbounded calculi.

Surprisingly, despite the  ever-growing
importance of the functional calculus for sectorial operators
since its inception by McIntosh \cite{McI86} in 1986 and subsequently of 
related calculi on strips, parabolas and other regions, 
there 
have been only a few authors (most notable: deLaubenfels \cite{deLaubenfels1995})
showing a strong will to
operate  with a reasonably abstract definition of a functional
calculus.


The first attempt, to the best of our knowlegde, to not
just give an axiomatic definition of an unbounded  functional 
calculus but
also to develop the associated abstract theory 
is due to the author of the present article. It was 
published in \cite{Haa05b} and then incorporated in 
and made widely known through the book \cite{HaaseFC}.  However, 
although not without merit and already quite abstract,
it has some shortcomings, which we adress in the following.

A  {\em first shortcoming} is that  
the definition of a functional calculus given there is intimately 
tied to a construction, algebraic extension by regularization. 
At the time, when the book was written, 
this was natural: algebraic extension
is a central tool, an elegant and easily manageable way of
reducing an unbounded calculus to its bounded part. Nevertheless,
from an advanced point of view it should be  obvious that 
a definition based on a specific construction 
cannot be regarded the definite answer to the axiomatization problem.

A {\em second shortcoming} of the framework from \cite{HaaseFC} 
is that only commutative algebras $\calF$ were allowed. 
Admittedly, the field of applications up to now 
almost exclusively involves
commutative algebras (algebras of scalar functions). But
in the future, genuinely non-commutative situations like
functional calculi arising from Lie group representations  will
become more and more important. Hence, there is a desire
to have a setting that does not rely on commutativity. 

Finally, a {\em third shortcoming} of the approach from 
\cite{HaaseFC} is that  topological ways of extending a functional
calculus are disregarded. (This is of course not a failure of
the axiomatic framing itself, but of its theoretical elaboration.) 
Actually, the need for such extensions had been formulated already in
\cite{Haa05b} and there was also a somewhat halfhearted attempt 
to provide them,  but that did not find much resonance.

With the present article, we are making a new attempt to 
find an adequate axiomatization of the notion of an (unbounded)
functional calculus and to develop its theory while 
avoiding the named shortcomings.

\bigskip

The paper is divided into two parts. The first part
is devoted to  the elaboration of the theory. As such, it is quite
abstract and sometimes technical (due to the lack of commutativity
assumptions). The second illustrates
the theory in some familiar situations, but with a stress on 
formerly unknown aspects, mainly regarding 
topological extensions. A reader chiefly interested in the second part
may safely skip the technical sections of the first part for the
time being and only revert to them when necessary. In the following
we give a short synopsis of the two parts.

\medskip
The first part starts with  the axioms of a calculus 
and their immediate consequences (Section \ref{s.afc}) and then proceeds
with the introduction of basic auxiliary notions like 
{\em determination}, {\em algebraic cores} (Section \ref{s.det}), 
and {\em anchor sets} (Section \ref{s.anc}).
The main theoretical problem  here consists in providing
criteria ensuring that an anchor set is actually determining. 
Whereas this is almost trivially true in a commutative situation (Theorem \ref{anc.t.com}),
some work is necessary to find such criteria without commutativity
(Theorem \ref{anc.t.main}). 
This dichotomy permeates also the subsequent
Section \ref{s.urc}, where the problems of uniqueness and
compatibility
are addressed. 

In Section \ref{s.ext} we discuss the {\em algebraic extension}. 
Surprisingly, 
no commutativity hypothesis whatsoever is needed to make
algebraic extension work (Theorems \ref{ext.t.ext} and
\ref{ext.t.ancgen}).  However, compatibility of successive extensions cannot be guaranteed
without additional assumptions (Theorem \ref{ext.t.succ-comp}).

In Section \ref{s.api} we briefly touch upon {\em approximate identities}.
This had been missed out in \cite{HaaseFC}, a first abstract result
was given by Clark \cite{Clark2009}. In Section \ref{s.dua} we introduce
the concept of a {\em dual calculus}.  In Section \ref{s.top} we discuss
{\em topological extensions}. This concludes the first part.

\medskip
In the second part, we illustrate the theory with some
familiar examples: sectorial operators (Section \ref{s.sec}), 
semigroup generators (Section \ref{s.sgr}) and 
normal operators (Section \ref{s.spt}). We provide new
topological extension theorems for the sectorial calculus
(Section \ref{s.sectop}) and the Hille--Phillips calculus 
(Section \ref{sgr.s.topext}).
In the sectorial case, we show how these topological extensions
covers calculi defined in the literature like 
the Stieltjes calculus and the  Hirsch calculus (Section \ref{s.hir}).
For generators of bounded semigroups we show that  the Hille--Phillips
calculus  is included in a certain topological extension 
of the sectorial calculus (Section \ref{sgr.s.sec-sgr}). 
In the Section \ref{s.spt}  on  normal operators on Hilbert space we 
report on a consistent functional calculus approach to the spectral
theorem (elaborated in \cite{Haase2020bpre}). 

\subsection*{Notation and Terminology}

We use the letters $X,Y, \dots$ generically to denote Banach spaces.
By default and unless otherwise stated, the scalar field is $\K =\C$.
The space of bounded linear operators from $X$ to $Y$ is denoted by
$\BL(X;Y)$, and $\BL(X)$ if $X= Y$.

A subset $\calD
\subseteq \BL(X)$ is called {\emdf point-separating} if
\[ \bigcap_{D\in \calD} \ker(D) = \{0\}.
\]
A (closed) linear relation in $X$ is any (closed) 
linear subspace $A\subseteq X \oplus X$. Linear relations
are called {\em multi-valued operators} in \cite[Appendix A]{HaaseFC}, and we use
freely the definitions
and results from that reference. In particular, we say
that a bounded operator $T$ {\emdf commutes} with
a linear relation $A$ if $TA \subseteq AT$, which is equivalent
to 
\[ (x,y) \in A \dann (Tx, Ty) \in A.
\]
If $\calE$ is a (multiplicative) semigroup, we denote its {\emdf center}
by
\beq\label{intro.e.center} 
\cnt(\calE) = \{ d\in \calE \suchthat \forall \, e\in \calE : de =
ed\}.
\eeq
If $\calF$ is a semigroup and $\calE\subseteq \calF$ is any subset, we
shall frequently use the notation
\[ [f]_\calE := \{ e\in \calE \suchthat ef \in \calE\}
 \qquad (f\in \calF).
\]

\part{Abstract Theory}

In the first part of this article, we treat the theory of functional
calculus in an abstract, axiomatic fashion. We aim at generality,
 in particular we do not make any standing commutativity assumption.
However, we emphasize that commutativity of the algebras 
greatly  simplifies theory and proofs.

\section{Axioms for Functional Calculi}\label{s.afc}

Let $\calF$ be an algebra with a unit element $\car$ 
and let $X$ be a Banach space.  A mapping
\[ \Phi: \calF \to  \Clo(X)
\]
from $\calF$ to the set of closed operators on $X$
is called a {\emdf proto-calculus} (or: 
$\calF$-proto-calculus) on $X$ 
if the following axioms are satisfied ($f,\: g\in \calF$, $\lambda \in \C$):
\begin{aufzi}
\item[\quad(FC1)] $\Phi(\car) = \Id$.
\item[\quad (FC2)] $\lambda \Phi(f) \subseteq \Phi(\lambda f)$
\quad  and \quad $\Phi(f) + \Phi(g) \subseteq \Phi(f +g)$.
\item[\quad (FC3)] $\Phi(f)\Phi(g) \subseteq \Phi(fg)$\quad and\quad
\[ 
\dom(\Phi(f)\Phi(g)) = \dom(\Phi(g)) \cap \dom(\Phi(fg)).
\]
\end{aufzi}
A proto-calculus $\Phi: \calF \to \Clo(X)$ 
is called a {\emdf calculus} if the following fourth axiom is
satisfied:
\begin{aufzi}
\item[\quad(FC4)] 
The set $\bdd(\calF, \Phi)$ of $\Phi$-bounded
elements is determining $\Phi$ on $\calF$.
\end{aufzi}
Here, an element $f\in \calF$ is called {\emdf $\Phi$-bounded}
if $\Phi(f)\in \BL(X)$, and the set of $\Phi$-bounded elements is
\[ \bdd(\Phi) := \bdd(\calF,\Phi) := \{ f\in \calF \suchthat \Phi(f) \in \BL(X)\} = \Phi^{-1}(\BL(X)).
\]

\medskip
The terminology and meaning of Axiom (FC4) shall be explained in
Section \ref{s.det}  below.  For the time  being we only 
suppose that $\Phi: \calF \to \Clo(X)$ is
a proto-calculus.
The following theorem summarizes its basic properties.

\begin{thm}\label{afc.t.pro-cal}
Let $\Phi: \calF \to \Clo(X)$ be a proto-calculus 
on a Banach space $X$. 
Then the following assertions hold ($f,\:g \in \calF$, $\lambda \in \C$):
\begin{aufzi}
\item If $\lambda \neq 0$ or $\Phi(f) \in \BL(X)$ then $\Phi(\lambda f) = \lambda \Phi(f)$.
\item If $\Phi(g)\in \BL(X)$ then
\[  \Phi(f) + \Phi(g) = \Phi(f+g)\quad \text{and}\quad \Phi(f)\Phi(g)
= \Phi(fg).
\]
\item If $fg= \car$ then $\Phi(g)$ is injective and
  $\Phi(g)^{-1} \subseteq \Phi(f)$. If, in addition, $fg=gf$, 
then $\Phi(g)^{-1} =\Phi(f)$. 

\item The set $\bdd(\calF,\Phi)$ of $\Phi$-bounded elements is a unital
  subalgebra of $\calF$ and 
\[ \Phi: \bdd(\calF,\Phi) \to \BL(X)
\]
is an algebra homomorphism. 
 \end{aufzi}
\end{thm}

\begin{proof}
a)\ One has 
\[ \Phi(f) = \Phi(\lambda^{-1} \lambda f) \supseteq 
\lambda^{-1} \Phi(\lambda f)\supseteq \lambda^{-1} \lambda \Phi(f)
=\Phi(f).
\]
Hence, all inclusions are equalities, and the assertion follows.

\prfnoi
b)\ By Axiom (FC2) and a)
\begin{align*}
 \Phi(f) &= \Phi(f +g - g) \supseteq \Phi(f+g) + \Phi(-g)
= \Phi(f+g) - \Phi(g) 
\\ & \supseteq \Phi(f) + \Phi(g) - \Phi(g)
= \Phi(f).
\end{align*}
Hence, all inclusions are equalities and the first  assertion in b)
follows. For the second, note that by Axiom (FC3)
$\Phi(f)\Phi(g) \subseteq \Phi(fg)$ with 
\[ \dom(\Phi(f)\Phi(g)) = \dom(\Phi(g)) \cap \dom(\Phi(fg))
=
\dom(\Phi(fg)),
\]
hence we are done. 

\prfnoi
c)\ By (FC3), if $fg= \car$ then $\Phi(f)\Phi(g) \subseteq \Phi(fg) = \Phi(\car)
= \Id$. Hence, $\Phi(g)$ is injective and $\Phi(f)\supseteq
\Phi(g)^{-1}$. If $fg= gf$, by symmetry 
$\Phi(f)$ is injective too, and  $\Phi(g) \supseteq
\Phi(f)^{-1}$. This yields $\Phi(f) = \Phi(g)^{-1}$ as desired. 

\prfnoi
d) follows directly from b). 
\end{proof}

\section{Determination}\label{s.det}

We shall write 
\beq\label{det.e.[f]}   [f]_\calE := \{ e\in \calE \suchthat ef\in \calE\}
\eeq
whenever $\calF$ is any multiplicative semigroup, $\calE \subseteq
\calF$ and $f\in \calF$. 
In our context, $\calF$ shall always be an
algebra. 

\medskip
Given a proto-calculus $\Phi: \calF \to \Clo(X)$ and $f\in \calF$ the 
set of its {\emdf $\Phi$-regularizers} is
\[ \reg(f, \Phi) := [f]_{\bdd(\calF, \Phi)} = \{ e\in \calF \suchthat e, \, 
ef \in \bdd(\calF,\Phi)\}.
\]
By Theorem \ref{afc.t.pro-cal}, $\reg(f, \Phi)$ is a left ideal 
in $\bdd(\calF,\Phi)$. The elements of the set
\[ \reg(\Phi):=  \reg(\calF, \Phi) := \bigcap_{f\in \calF} \reg(f,\Phi)
\]
are called   {\emdf universal regularizers}. Of course, it may
happen that $e= 0$ is the only universal regularizer.

\begin{rem}\label{det.r.reg-old}
The definition of a  regularizer here differs from and is more
general than the one given in \cite{HaaseFC}. 
It has been argued in \cite{Haase2008pre} (eventually published as \cite{Haase2017}) 
that  such a relaxation of terminology is useful.
\end{rem}

Let $f\in \calF$. A  subset $\calM \subseteq \bdd(\calF,\Phi)$  
is said to {\emdf determine $\Phi(f)$} if
\begin{equation}\label{afc.e.determining}  
\Phi(f)x = y \quad \Gdw\quad 
\forall e \in
\calM  \cap \reg(f, \Phi):\,\,  \Phi(ef)x = \Phi(e)y
\end{equation}
for all $x,\: y\in X$. And $\calM$ is said to {\emdf strongly
  determine}
$\Phi(f)$ if the set $[f]_{\calM}$ determines $\Phi(f)$, i.e. if
\begin{equation}\label{afc.e.strong-determining}  
\Phi(f)x = y \quad \Gdw\quad 
\forall e \in [f]_{\calM} :\,\,  \Phi(ef)x = \Phi(e)y
\end{equation}
for all $x,\: y\in X$.

\begin{rems}\label{det.r.det-super}
\begin{aufziii}
\item Although very useful, the terminology ``$\calM$ determines
  $\Phi(f)$'' is to be used with caution: there might be $g\in
  \calF$ with $f\neq g$ and  $\Phi(g) = \Phi(f)$ and such that $\Phi(g)$ is not determined by
$\calM$ in the above sense. (In other words, the expression ``$\Phi(f)$'' has
  to be interpreted symbolically here, and not as a mathematical object.) 
With this caveat in mind, there should be little danger of
 confusion. However, if we want to be completely accurate, we shall
use the alternative formulation ``{\emdf $\calM$ determines $\Phi$ at
  $f$}''.

\item 
We observe  that in  
both equivalences \eqref{afc.e.determining}  and \eqref{afc.e.strong-determining}  
only the implication ``$\Leftarrow$'' is relevant, as 
the implication ``$\dann$'' simply means
$\Phi(e)\Phi(f) \subseteq \Phi(ef)$  
for $e\in \calM\cap \reg(f,\Phi)$ or $e\in [f]_{\calM}$,
respectively; and this follows from Axiom (FC3).

An immediate consequence of this observation is that a strongly determining
set is determining.
\end{aufziii}
\end{rems}

A set $\calE \subseteq \bdd(\calF,\Phi)$ 
is said to be {\emdf determining $\Phi$ on} or is {\emdf
  $\Phi$-determining for} $\calF$ 
if it determines $\Phi(f)$ for each $f\in \calF$.
(If $\calF$ is understood, reference to it
is often dropped, and one simply speaks of $\Phi$-determining sets.)
Note that by Remark \ref{det.r.det-super}, {\em any superset of a determining
set is again determining}.

A subset $\calE$ of $\bdd(\calF, \Phi)$ is called an {\emdf algebraic core}
for $\Phi$ on  $\calF$ if $\calE$  strongly determines 
$\Phi(f)$ for each $f\in \calF$. 
In this terminology, Axiom (FC4) simply requires $\bdd(\calF, \Phi)$
to be an algebraic core for $\Phi$ on $\calF$.  Again,
by Remark \ref{det.r.det-super}  {\em any superset of an algebraic core
is also an algebraic core}.

\begin{exa}[Bounded (Proto-)Calculi] 
A proto-calculus $\Phi: \calF \to \Clo(X)$ with 
$\calF = \bdd(\calF, \Phi)$ is nothing else than 
a unital algebra representation $\Phi: \calF \to \BL(X)$. 
By unitality, Axiom (FC4) is automatically satisfied in
this situation. So each unital representation by bounded
operators is a calculus.
\end{exa}

\section{Anchor Sets}\label{s.anc}



Let $\calE$ be a set and let $\Phi: \calE \to \BL(X)$ be any mapping.
An element $e\in \calE$ is called an
{\emdf anchor element} if $\Phi(e)$ is injective.
More generally, a subset $\calM \subseteq \calE$
is called an {\emdf anchor set} if $\calM \neq \leer$ and 
the set $\{ \Phi(e) \suchthat e\in \calM\}$ is point-separating, i.e., if 
\[ \bigcap_{e\in \calM} \ker(\Phi(e)) = \{0\}.
\]
If $\calF$ is a semigroup, then we  say that $f\in \calF$ is {\emdf anchored} in $\calE$
(with respect to  $\Phi$) if the set $[f]_\calE$
is an anchor set. And we call $\calF$ {\emdf anchored in $\calE$} if 
each $f\in \calF$ is anchored in $\calE$.  If we want to 
stress the dependence on $\Phi$, we shall speak of 
{\emdf $\Phi$-anchor elements/sets} and of elements/sets being {\emdf
  $\Phi$-anchored} in  $\calE$.

\medskip
Suppose that $\Phi: \calF \to \Clo(X)$ is a proto-calculus
and $f\in \calF$. Any set $\calM \subseteq \bdd(\Phi, \calF)$ which
determines $\Phi(f)$ must be an anchor set, just because
$\Phi(f)$ is an operator and not just a linear relation. 
The converse question, i.e., whether a particular anchor
set does actually determine $\Phi(f)$ is the subject
of the present section. We start with a simple case. 

\begin{thm}\label{anc.t.com}
Let $\Phi: \calF \to \Clo(X)$ be a calculus and let
$f\in\calF$. If $\calF$ is commutative and
$\calE \subseteq \reg(f, \Phi)$ is an anchor set, then 
$\calE$ determines $\Phi(f)$. 
\end{thm}

\begin{proof}
Let $x,y\in X$ such that $\Phi(ef)x = \Phi(e)y$ for
all $e\in \calE$. Then for all $e\in \calE$ and $g\in \reg(\Phi, \calF)$ we have
\begin{align*}
 \Phi(e)\Phi(gf)x & = \Phi(egf)x =
\Phi(gef)x = \Phi(g) \Phi(ef)x = 
\Phi(g) \Phi(e)y = \Phi(ge)y \\ & = \Phi(eg)y = 
\Phi(e)\Phi(g)y.
\end{align*}
Since $\calE$ is an anchor set, it follows that 
$\Phi(gf)x = \Phi(g)y$ for all $g\in \reg(\Phi, \calF)$.
Since $\Phi$ is a calculus, this implies $\Phi(f)x = y$. 
\end{proof}

Without commutativity,
things become much more complicated. It actually
may come as a surprise that the following result
is true in general.

\begin{thm}\label{anc.t.main}
Let $\Phi: \calF \to \Clo(X)$ be  a calculus, and let
$\calE  \subseteq \bdd(\calF,\Phi)$ be such that 
$\bdd(\calF, \Phi)$ is anchored in $\calE$. 
Then $\calE$ determines $\Phi$ on $\calF$. If,
in addition, $\calE$ is multiplicative, then 
$\calE$ is an algebraic core for $\Phi$. 
\end{thm}

For the proof of  Theorem \ref{anc.t.main} we need
some auxiliary results about
determining sets and anchor sets. This is the subject
of the following lemma.

\begin{lem}\label{anc.l.pro-cal-deter}
Let $\Phi: \calF \to \Clo(X)$ be a proto-calculus, let $f\in \calF$,
let $\calM \subseteq \bdd(\calF,\Phi)$ and $\calB \subseteq \calF$.
Then the following assertions hold:
\begin{aufzi}
\item If $\calM$ determines $\Phi(f)$, then it is 
  an anchor set.  The converse holds if $f\in \bdd(\calF, \Phi)$.

\item If $\calM$ determines $\Phi(f)$ (is an anchor set) and
$\calM \subseteq \calN \subseteq \bdd(\calF, \Phi)$ then $\calN$ determines
  $\Phi(f)$ (is an anchor set).

\item If $\calM$ determines $\Phi(f)$ (is an anchor set)
and $\calN_g \subseteq
  \bdd(\calF, \Phi)$ is an anchor set for each $g\in \calM$, 
then also $\calN := \bigcup_{g\in \calM} \calN_g g$
is determines $\Phi(f)$ (is an anchor set).

\item Suppose that
$\calB \calM = \{ bg \suchthat b \in \calB,\, g\in \calM\} \subseteq
\bdd(\Phi, \calF)$. If $\calB \calM$ 
is an anchor set, then so is $\calM$. If 
$\calM \subseteq \reg(f, \Phi)$ and $\calB\calM$
determines $\Phi(f)$, then so does $\calM$.


\item  If $T\in \BL(X)$  commutes with all operators  $\Phi(e)$ and
$\Phi(ef)$ for  $e\in \calM$, and $\calM$ determines 
$\Phi(f)$, then  $T$ commutes with $\Phi(f)$. 
\end{aufzi}
\end{lem}

\begin{proof}
a)\  The first assertion follows from \eqref{afc.e.determining} and the fact that 
$\Phi(f)$ is an operator and not just a linear relation.
For the second suppose that $\Phi(f) \in \BL(X)$ and $\calM$ is an
anchor set. Let $x, y\in X$ such that $\Phi(ef)x= \Phi(e)y$ for
all $e\in \calM$. Then, since $\Phi(f)$ is bounded, 
\[ \Phi(e)y = \Phi(ef)x = \Phi(e)\Phi(f)x.
\]
Since $\calM$ is an anchor set, it follows that $\Phi(f)x = y$. 

\prfnoi
b)\ is trivial (and has already been mentioned above).

\prfnoi
c)\ Suppose first that $\calM$ determines $\Phi(f)$
and
that $\Phi(egf)x = \Phi(eg)y$ for all $g\in \calM$ and
all $e\in \calN_g$ such that $eg\in \reg(f, \Phi)$. 
Fix $g\in \calM\cap \reg(f,\Phi)$. Then for each $e\in \calN_g$
we have $eg\in \reg(f,\Phi)$ and hence
\[ \Phi(e) \Phi(gf)x = \Phi(egf)x = \Phi(eg)y = \Phi(e)\Phi(g)y
\]
Since $\calN_g$ is an anchor set, it follows that $\Phi(gf)x= \Phi(g)y$. Since 
$\calM$ determines $\Phi(f)$, it follows that $\Phi(f)x = y$ as claimed.

Similary (and even more easily) one proves that 
$\calN$ is an anchor set if $\calM$ is one.

\prfnoi
d)\ Suppose that  $\calB \calM$ is an anchor set. For each $b\in \calB$
and $f\in \calM$ we have $\Phi(bf) = \Phi(b)\Phi(f)$ and hence
$\ker(\Phi(f)) \subseteq \ker(\Phi(bf))$. It follows readily that $\calM$ is an anchor set.

Now suppose that $\calM \subseteq \reg(f, \Phi)$ and that $\calB
\calM$ determines $\Phi(f)$. Let $x,y\in X$ 
such that  
$\Phi(gf)x = \Phi(g)y$ for all $g\in \calM$. Then
$\Phi(bgf)x = \Phi(b) \Phi(gf)x= \Phi(b)\Phi(g)y= \Phi(bg)y$ for all
$b\in \calB$ and $g\in \calM$. Hence,  by hypothesis, $\Phi(f)x = y$.

\vanish{
\prfnoi
e)\ Since $\calB$ determines $\Phi(f)$ and $\calM$ is an anchor
set, $\calM \calB$ determines $\Phi(f)$. (This can be seen as an
application of c).) By hypothesis and b) above, 
$\calC \calM$ determines $\Phi(f)$. Then, d) implies that 
$\calM$ determines $\Phi(f)$. 
}

\prfnoi
e)\ Suppose that $\Phi(f)x = y$. Then, for each $e\in \calM$, 
\[ \Phi(ef)Tx = T\Phi(ef)x = T\Phi(e)y = \Phi(e)Ty.
\]
Since $\calM$ determines $\Phi(f)$,  
$\Phi(f)Tx = Ty$ as claimed.
\end{proof}

Let us mention that assertion c) is
by far the most important
part of  Lemma \ref{anc.l.pro-cal-deter}. 
As a first  consequence, we
note the following result.

\begin{prop}\label{anc.p.EM}
Let $\Phi: \calF \to \Clo(X)$ be a proto-calculus,
let $f\in \calF$, and let
$\calM, \calE\subseteq \bdd(\Phi, \calF)$
such that $\calM$ determines $\Phi(f)$. 
Suppose that 
one of the following condititions holds: 
\begin{aufziii}
\item $\calM$ is anchored in $\calE$.
\item For each $g\in \calM$ there is an anchor set
$\calN_g$ such that $\calN_g g \subseteq \calB \calE$,  where
$\calB := \{ h \in \calF \suchthat h \calE \subseteq \bdd(\Phi,
\calF)\}$.
\end{aufziii}
Then $\calE$ determines $\Phi(f)$. 
\end{prop}

\begin{proof}
1)\ For each $g\in \calM' := \calM \cap \reg(f, \Phi)$ 
let $\calN_g := [g]_\calE$. By hypothesis, this is an anchor set. 
Hence, by Lemma \ref{anc.l.pro-cal-deter}.c), 
$\calN := \bigcup_{g\in \calM'} \calN_g g$ determines
$\Phi(f)$. As $\calN \subseteq \calE \cap \reg(f, \Phi)$, 
also  $\calE$ determines $\Phi(f)$.

\prfnoi
2)\ By hypothesis, $\calE$ is an anchor set and $\calM$ determines
$\Phi(f)$. Hence, by Lemma \ref{anc.l.pro-cal-deter}.c), 
$\calE \calM = \bigcup_{g \in \calM} \calE g$ determines $\Phi(f)$.
Since $\calE \calM\subseteq \calB \calE$, also the latter set
determines $\Phi(f)$. Lemma \ref{anc.l.pro-cal-deter}.d) then implies
that $\calE$ determines $\Phi(f)$.
\end{proof}

\begin{rem}
Theorem \ref{anc.t.com} is a 
corollary of Proposition \ref{anc.p.EM}. (Apply
part 2) with  $\calN_g = \calE$ for each 
$g\in \calM := \reg(\Phi, \calF)$.) 

Actually, we realize that one may replace \
the overall commutativity assumption
on $\calF$ by a weaker condition, e.g.: 
{\em $[f]_\calE \cap \{ e\in \calE \suchthat e \calM \subseteq 
\calM e\}$
is an anchor set.} 
\end{rem}

Finally, we are going to  prove Theorem \ref{anc.t.main}.

\begin{proof}[Proof of Theorem \ref{anc.t.main}]
Write $\calM := \bdd(\Phi, \calF)$ and
fix $f\in \calF$. Then $\calM$ determines $\Phi(f)$,
since $\Phi$ is a calculus, and $\calM$ is anchored in $\calE$,
by assumption. Hence, by 
Proposition \ref{anc.p.EM}.1), $\calE$ determines $\Phi(f)$. 

\prfnoi
For the second assertion,  
suppose in addition that $\calE$ is multiplicative. 
Fix $g\in \reg(f, \Phi)$. Then, as above, 
$[g]_\calE$ is an anchor set. Also, for each $e\in [g]_\calE$
one has $egf\in \calM$ and hence $[egf]_\calE$ is also an anchor
set. It follows that $\calN_g := \bigcup_{e\in [g]_\calE} [egf]_\calE
e$ is an anchor set. By Lemma \ref{anc.l.pro-cal-deter}.c), 
the set 
\[ \calN := \bigcup_{g\in \reg(f, \Phi)} \calN_g g
\] 
determines $\Phi(f)$. But $\calE$ is multiplicative, and therefore
 \[ \calN = \bigcup_{g\in \reg(f, \Phi)} \calN_g g
= \bigcup_{g\in \reg(f, \Phi)} 
\bigcup_{e\in [g]_\calE} [egf]_\calE eg  \subseteq [f]_\calE.
\]
Hence, also  $[f]_\calE$ determines $\Phi(f)$.
\end{proof}


\section{Uniqueness, Restriction, Compatibility}\label{s.urc}

Let  $\Phi_1, \Phi_2: \calF \to \Clo(X)$ be calculi, let $f\in
\calF$ and suppose that $\calM\subseteq \calF$ determines both calculi
$\Phi_1$ and $\Phi_2$  at $f$ (cf. Remark \ref{det.r.det-super}). 
Suppose further that $\Phi_1$ and $\Phi_2$ agree on $\calM$. Can one
conclude that $\Phi_1(f) = \Phi_2(f)$?

A moment's reflection reveals that the answer might
be ``no'' in general. The reason is that we
do not know whether $\Phi_1$ and $\Phi_2$ coincide on products $ef$
for $e\in \calM$. This is the original motivation for introducing the concept of
strong determination.

\begin{lem}\label{urc.l.uni}
Suppose that $\Phi_1, \Phi_2 : \calF \to \Clo(X)$ are two
proto-calculi and $\calE \subseteq \calF$  such that
$\Phi_1\res{\calE} = \Phi_2\res{\calE}$.  If $\calE$ is an algebraic core
for $\Phi_2$  then
\[ \Phi_1(f) \subseteq \Phi_2(f)\qquad \text{for each $f\in \calF$}.
\]
In particular, if $\calE$ is an algebraic core for both calculi, then
$\Phi_1 = \Phi_2$.
\end{lem}

\begin{proof}
Let $f\in \calF$. And suppose that $x,y\in X$ are such that
$\Phi_1(f)x= y$. Then for every $e\in [f]_\calE$ we have
\[ \Phi_1(ef)x = \Phi_1(e)y.
\]
This is the same as $\Phi_2(ef)x= \Phi_2(e)y$, as $\Phi_1$ and
$\Phi_2$ agree on $\calE$. By hypothesis, $[f]_\calE$ determines
$\Phi_2(f)$, so it follows that $\Phi_2(f)x= y$. 
This shows $\Phi_1(f)\subseteq \Phi_2(f)$.
\end{proof}

Combining Lemma \ref{urc.l.uni} with Theorem \ref{anc.t.main} 
we arrive
at the following uniqueness statement.

\begin{thm}[Uniqueness]\label{urc.t.uni}
Let $\Phi_1, \Phi_2: \calF \to \Clo(X)$ be calculi. Suppose that there is $\calE \subseteq \calF$ with the following properties:
\begin{aufziii}
\item $\Phi_1(e) = \Phi_2(e) \in \BL(X)$ for all  $e\in \calE$.
\item $\calF$ is  anchored in $\calE$ (with respect to one/both calculi).
\end{aufziii}
Then $\Phi_1 = \Phi_2$. 
\end{thm}

\begin{proof}
By passing to
\[ \calE':= \bigcup_{n\ge 1} \calE^n = 
\{ e_1\cdots e_n \suchthat 
n \in \N, \, e_1, \dots, e_n \in \calE\},
\]
the multiplicative semigroup generated
by $\calE$ in $\calF$, we may suppose that $\calE$ is multiplicative.
Then Theorem  \ref{anc.t.main} yields that
$\calE$ is an algebraic core for both calculi. Hence, Lemma \ref{urc.l.uni}
yields $\Phi_1 = \Phi_2$.
\end{proof}

\medskip

\subsection{Pull-Back and Restriction of a Calculus}

Suppose that $\calF$ is a unital algebra
and $\Phi: \calF \to \Clo(X)$ is a proto-calculus. 
Further, let $\calG$ be a unital algebra and
$\eta: \calG \to \calF$ a unital algebra homomorphism.
Then the mapping 
\[ \eta^*\Phi: \calG \to \Clo(X),\qquad (\eta^*\Phi)(g) :=
\Phi(\eta(g))
\]
is called  the {\emdf pull-back} of $\Phi$ {\emdf along} $\eta$. 
It is easy to see that, in general, $\eta^*\Phi$ is a proto-calculus
as well.
A special case occurs if $\calG$ is a subalgebra 
of $\calF$ and $\eta$ is the inclusion mapping. 
Then $\eta^*\Phi = \Phi\res{\calG}$ is just the {\emdf restriction}
of $\Phi$ to $\calG$.

\begin{lem}\label{urc.l.pull-back}
Let $\calF$ be a unital algebra and $\Phi: \calF \to
\Clo(X)$ a proto-calculus. Furthermore, let $\calG$ be a unital
algebra, $\eta: \calG \to \calF$ a unital homomorphism, 
$\calE \subseteq \bdd(\eta^*\Phi, \calG)$, and $g\in \calG$.
Then the following assertions hold. 
\begin{aufzi}
\item $\bdd(\eta^*\Phi, \calG) = \eta^{-1}( \bdd(\Phi, \calF))$,\\
\quad $\eta (\bdd(\eta^*\Phi, \calG)) = \bdd(\Phi, \calF) \cap
\eta(\calG) = \bdd(\Phi\res{\eta(\calG)}, \eta(\calG))$.
\item $\calE$ is an $\eta^*\Phi$-anchor
if and only if $\eta(\calE)$ is a $\Phi$-anchor.

\item $\eta( [g]_\calE) \subseteq  [\eta(g)]_{\eta(\calE)}$.

\item $\calE$
determines $\eta^*\Phi$ at $g$ if and only if
$\eta(\calE)$ determines $\Phi$ at $\eta(g)$.
\item If $\calE$ is an algebraic core for 
$\eta^*\Phi$ then  $\eta(\calE)$ is an algebraic
core for $\Phi\res{\eta(\calG)}$.
\end{aufzi}
\end{lem}

\begin{proof}
a), b), and c) follow directly from the definition of $\eta^*\Phi$.

\prfnoi
d) This follows since for $x,y\in X$ and $e \in \calE$ 
\[ \Phi(\eta(e) \eta(g))x= \Phi(\eta(e))y
\quad\gdw
\quad (\eta^*\Phi)(eg)x = (\eta^*\Phi)(e)y.
\]
e)  This follows from c) and d). 
\end{proof}

We have already remarked that $\eta^*\Phi$ is a proto-calculus,
whenever $\Phi$ is one. The following example shows that,
even if $\Phi$ is a calculus,   $\eta^*\Phi$ need not be one.

\begin{exa}
Let $A$ be an unbounded operator with non-empty resolvent set $\resolv(A)$.
Let $\calF$ be the algebra of all rational functions
with poles in $\resolv(A)$ and let $\Phi$ be the natural
calculus (as described, e.g., in \cite[Appendix A.6]{HaaseFC}). 
Let $\calG$ be the algebra of polynomial functions
 and $\eta: \calG \to \calF$ the inclusion mapping. Then 
$\eta^*\Phi = \Phi\res{\calG}$ is simply the restriction of 
$\Phi$ to $\calG$. And this is not a calculus,
as the only functions in $\calG$ that yield bounded operators
are the constant ones. 
\end{exa}

We say that $\eta$ is {\emdf $\Phi$-regular} if
$\eta^*\Phi$ is a calculus. And a subalgebra $\calG$ 
of $\calF$ is called {\emdf $\Phi$-regular}, if 
the restriction of $\Phi$ to $\calG$ is a calculus, i.e.,
if the inclusion mapping is $\Phi$-regular.

\begin{cor}\label{urc.c.Phi-reg}
In the situation of Lemma \ref{urc.l.pull-back}, 
the following statements
are equivalent:
\begin{aufzii}
\item $\eta$ is a $\Phi$-regular mapping, i.e., $\eta^*\Phi$ is a calculus.
\item $\eta(\calG)$ is a $\Phi$-regular subalgebra of $\calF$, i.e.,
$\Phi\res{\eta(\calG)}$ is a calculus. 
\end{aufzii}
\end{cor}

\begin{proof}
This follows from a) and d) of Lemma \ref{urc.l.pull-back} with
$\calE = \bdd(\eta^*\Phi)$.
\end{proof}

\vanish{
\begin{lem}\label{urc.l.Phi-reg}
Let  $\calF, \calG$ be unital algebras,
let $\Phi: \calF \to \Clo(X)$ be a calculus, and 
let $\eta: \calG \to \calF$ be a unital algebra homomorphism. 
Define $\calE := \eta(\calG) \cap \bdd(\calF, \Phi)$. 
Then $\calE = \eta(\bdd(\eta^*\Phi, \calG))$ and 
the following assertions are equivalent:
\begin{aufzii}
\item The mapping $\eta$ is a$\Phi$-regular
\item  The subalgebra $\eta(\calG)$ is $\Phi$-regular.
\item 
The set
\[ \calM := \{ g\in \bdd(\calF, \Phi) \suchthat 
[g]_\calE \,\,\text{is an anchor set}\}
\]
determines $\Phi$ on $\eta(\calG)$. 
\end{aufzii}
This is the case, e.g., if $\bdd(\calF, \Phi) \subseteq \eta(\calG)$. 
\end{lem}

\begin{proof}
(i)$\gdw$(ii): Note that $\calE = \bdd(\eta(\calG), \Phi)$ and 
\[ \calE':= \bdd(\calG, \eta^*\Phi) = \eta^{-1} \bdd(\calF, \Phi) 
= \eta^{-1}\calE.
\]
A short argument reveals: $\calE$ determines $\Phi$ on $\eta(\calG)$ 
if and only if $\calE'$ determines $\eta^*\Phi$ on $\calG$.
The first is equivalent to $\calG$ being $\Phi$-regular;
the latter is equivalent to $\eta$ being $\Phi$-regular. 

\prfnoi
(ii)$\dann$(iii):  
For each $e\in \calE$ one has $[e]_\calE = \calE$ (since $\calE$ is a
subalgebra of $\calF$) and this is an anchor set (since $\car \in
\calE$). It follows that $\calE \subseteq \calM$ and hence the claim.

\prfnoi
(iii)$\dann$(ii):  By definition, $\calM$ is anchored in $\calE$. So 
by Theorem \ref{anc.t.EM}.a) it follows that $\calE$ determines $\Phi$
on $\eta(\calG)$. 
\end{proof}

}

It seems that, in general, one cannot say much more.
However, here is an interesting special case, when one can simplify 
assumptions.

\begin{thm}\label{urc.t.pull-back-com}
Let $\calF$ be a commutative unital algebra and $\Phi: \calF \to
\Clo(X)$ a calculus. Let $\calG$ be a unital subalgebra of $\calF$
such that $\reg(g, \Phi)\cap  \calG$ is an anchor set for each
$g\in \calG$. Then $\calG$ is $\Phi$-regular, i.e., $\Phi\res{\calG}$ is a calculus.
\end{thm}

\begin{proof}
This follows immediately from Theorem \ref{anc.t.com}.
\end{proof}

In view of Lemma \ref{urc.l.pull-back} we obtain 
the following consequence.

\begin{cor}\label{urc.c.com}
Let $\calF$ be a commutative unital algebra and $\Phi: \calF \to
\Clo(X)$ a calculus. Furthermore, let $\calG$ be a unital
algebra and $\eta: \calG \to \calF$ a unital homomorphism. 
Then $\eta$ is $\Phi$-regular if and only if for each $g\in \calG$
the set $\reg(\eta(g),\Phi) \cap \eta(\calG)$ is an
anchor set.
\end{cor}

\vanish{
\begin{proof}
In view of Lemma \ref{urc.l.pull-back} 
we may suppose without loss of generality
that $\calG \subseteq \calF$ and $\eta$ is the inclusion mapping.  
Fix $f\in \calG$ and suppose that $x,y\in X$ are such that
$\Phi(ef)x= \Phi(e)y$ for all $e\in \calG \cap \reg(f,\Phi)$.
For each $g\in \reg(f, \Phi)$ one has
\begin{align*}
  \Phi(e)\Phi(gf)x &= \Phi(egf)x = \Phi(gef)x = \Phi(g)\Phi(ef)x
= \Phi(g)\Phi(e)y
\\ & = \Phi(eg)y = \Phi(eg)y= \Phi(e)\Phi(g)y.
\end{align*}
By assumption, $\reg(f, \Phi)\cap \calG$ is an anchor set, hence it follows
that  $\Phi(gf)x= \Phi(g)y$. As $\reg(f, \Phi)$ determines
$\Phi(f)$, we obtain $\Phi(f)x = y$ as claimed.
\end{proof}
}

\medskip

\subsection{Compatibility and Composition Rules}

Suppose one has set up a functional calculus $\Phi= ( f\mapsto f(A))$ for an
operator $A$ and a second functional calculus $\Psi= (g \mapsto g(B))$ for 
an operator $B$ which is of the form $B = f(A)$. Then one would expect
a ``composition rule'' of the form $g(B) = (g\nach f)(A)$. This amounts to
the identity $\Psi = \eta^*\Phi$, where $\eta = (g \mapsto g\nach
f)$ is an algebra homomorphism that links the domain algebras of the
two calculi. The following theorem, which basically is just a
combination of results obtained so far, 
yields criteria for this composition rule to hold true.

\begin{thm}\label{urc.t.com}
Let $\calF$ and $\calG$ be unital algebras and
$\eta: \calG \to \calF$ a unital algebra homomorphism. 
Furthermore, let $\Phi: \calF \to \Clo(X)$ and
$\Psi: \calG\to \Clo(X)$ be proto-calculi, and let 
$\calE$ be an algebraic core for $\Psi$ such that  
\[  \Phi(\eta(e)) = \Psi(e) \qquad (e\in \calE).
\]
Then  
the following statements are equivalent:
\begin{aufzii}
\item  $\Phi \nach \eta = \Psi$.
\item  $\eta(\calG)$ is a $\Phi$-regular subalgebra of $\calF$.
\item  $\eta(\calE)$ is an algebraic core for the restriction
of $\Phi$ to $\eta(\calG)$.
\end{aufzii}
Moreover, {\rm (i)-(iii)} hold true if, e.g.,
$\Phi$ is a calculus and $\calF$ is commutative.
\end{thm}

\begin{proof}
(i)$\dann$(iii): Since $\calE$ is an algebraic core
for $\Psi$ and, by (i), $\Psi = \eta^*\Phi$, 
the set $\eta(\calE)$ must be an algebraic core
for $\Phi$ on $\eta(\calG)$, by e) of Lemma \ref{urc.l.pull-back}.

\prfnoi
(iii)$\dann$(ii): If (iii) holds then (ii) follows a fortiori.

\prfnoi
(ii)$\dann$(i): If (ii) holds that $\eta^*\Phi$ is a calculus. 
Also, by hypothesis, $\Psi$ is a calculus. These calculi
agree on $\calE$, and this is an anchor (since it
is an algebraic core for $\Psi$). Hence, by the Uniqueness
Theorem \ref{urc.t.uni}, $\eta^*\Phi = \Psi$, i.e., (i). 

\prfnoi
Finally, suppose that $\Phi$ is a calculus and $\calF$ is
commutative. Let $g \in \calG$. Then $[g]_\calE$ is a  $\Psi$-anchor set.
Hence, $[\eta(g)]_{\eta(\calG)}$ is a $\Phi$-anchor set.
By Corollary \ref{urc.c.com}, $\eta$ is $\Phi$-regular, i.e., (ii). \end{proof}

See also  Theorem \ref{ext.t.succ-comp} below for 
more refined compatibility criteria.

\section{Algebraic Extension}\label{s.ext}

From now on, we suppose that $\calE$, $\calF$ and $\Phi$ are such that
\begin{aufziii}
\item $\calF$ is a unital algebra,
\item $\calE$ is a (not necessarily unital) subalgebra of $\calF$,
\item $\Phi: \calE \to \BL(X)$ is an algebra representation. 
\end{aufziii}

Our  goal is to
give conditions on $\calF$ such that 
a given representation $\Phi: \calE \to \BL(X)$ can be extended 
to an $\calF$-calculus in a unique way. 
A glance at Theorem \ref{anc.t.main} leads us to hope that it might be
helpful to require in addition to 1)--3)
also: 
\begin{aufziii}\setcounter{aufziii}{3}
\item Each $f\in \calF$ is anchored in $\calE$.
\end{aufziii}
The next result tells that under these assumptions
there is indeed a unique calculus on $\calF$ extending $\Phi$.

\begin{thm}[Extension Theorem]\label{ext.t.ext}
Let $\calF$ be a unital algebra and $\calE \subseteq \calF$ a
subalgebra. Furthermore, let $X$ be a Banach space and
let $\Phi: \calE \to \BL(X)$ be an algebra
homomorphism such that $\calF$ is anchored in $\calE$.
Then there is a unique calculus 
\[ \fourier{\Phi}:  \calF \to \Clo(X)
\]
such that $\fourier{\Phi}\res{\calE} = \Phi$. 
\end{thm}

\begin{proof}
Uniqueness follows directy from Theorem \ref{urc.t.uni}. Moreover, since
$\calE$ is multiplicative, for each $f\in \calF$ the set
$[f]_\calE$ must determine $\fourier{\Phi}(f)$.  Hence,
for the existence proof we have no other choice than  to {\em define}
\beq\label{afc.eq.ext-def} 
\fourier{\Phi}(f)x = y \quad \defiff\quad \forall\, e\in [f]_\calE
: \Phi(ef)x= \Phi(e)y
\eeq
for any $x,y\in X$ and $f\in \calF$.  Note that since $[f]_\calE$ is
an anchor set, $\fourier{\Phi}(f) \in \Clo(X)$.
It remains to show that
$\fourier{\Phi}$ extends $\Phi$ and satisfies (FC1)---(FC4).

\prfnoi
(FC1):\ By hypothesis,  $[\car]_\calE = \calE$ is an anchor set.
Hence, $x=y$ is equivalent to 
\[ \Phi(e)x= \Phi(e)y  \qquad
\text{for all $e\in \calE$},
\]
which, by definition \eqref{afc.eq.ext-def} 
is equivalent to $\fourier{\Phi}(\car)x= y$.

\prfnoi
Next, let us show that $\fourier{\Phi}$ extends $\Phi$.
To that end,  let $f\in \calE$. Then $[f]_\calE = \calE$, and
$\fourier{\Phi}(f)x = y$ is equivalent to 
\[ \Phi(e)y = \Phi(ef)x = \Phi(e)\Phi(f)x \quad \text{for all $e\in \calE$},
\]
which is equivalent to $y= \Phi(f)x$ (since $\calE$ is an anchor).

\prfnoi
(FC2): Let $\lambda \in \C$ and $f\in \calF$ and take $x,y\in X$ with
$\lambda\fourier{\Phi}(f)x = y$. We need to show that 
$\fourier{\Phi}(\lambda f)x = y$. This is clear if $\lambda = 0$.
If $\lambda \neq 0$ we find $\fourier{\Phi}(f)x = \lambda^{-1}y$ and
hence $\Phi(ef)x= \Phi(e)(\lambda^{-1}y)$, or better
\[ \Phi(e (\lambda f))x = \Phi(e)y
\]
for every $e\in [f]_\calE = [\lambda f]_\calE$. And the latter
statement
just tells that $\fourier{\Phi}(\lambda f)x = y$, as desired.

Now pick $f, g\in \calF$ and suppose that $\fourier{\Phi}(f)x= y$ and
$\fourier{\Phi}(g)x = z$. Take $h\in [f + g]_\calE$ and 
$e\in [hf]_\calE$. Then $eh\in [f]_\calE \cap [g]_\calE$ and hence
\[ \Phi(ehf)x = \Phi(eh)y \quad \text{and}\quad \Phi(ehg)x =
\Phi(eh)z.
\]
This yields
\[ \Phi(e) \Phi(h(f+g))x = \Phi(ehf)x + \Phi(ehg)x= 
\Phi(eh)(y+z) = \Phi(e) \Phi(h)(y+z).
\]
 Since $[hf]_\calE$ is a anchor set, it follows that 
\[ \Phi(h(f+g))x = \Phi(h)(z+y)
\]
and since $h\in [f+g]_\calE$ was arbitrary, we arrive at
$\fourier{\Phi}(f+g)x = y$.

\prfnoi
(FC3):  Let $\fourier{\Phi}(g)x = y$ and $\fourier{\Phi}(f)y = z$, and 
let $h\in [fg]_\calE$. Then for each $e\in [hfg]_\calE$ and
$e'\in[ehf]_\calE$
one has $e'e h\in [f]_\calE$ and $e'ehf \in [g]_\calE$ and hence
\[ \Phi(e') \Phi(e) \Phi(hfg)x = 
\Phi(e'e h fg)x = \Phi(e'ehf)y = \Phi( e'eh)z = \Phi(e')\Phi(e)\Phi(h)z.
\]
Since $[ehf]_\calE$ is anchor set,  $\Phi(e) \Phi(hfg)x = \Phi(e)\Phi(h)z$,
and since $[hfg]_\calE$ is an anchor set, $\Phi(hfg)x = \Phi(h)z$. All
in all we conclude that $\fourier{\Phi}(fg)x = z$. This proves
the inclusion
\[ \fourier{\Phi}(f) \fourier{\Phi}(g) \subseteq \fourier{\Phi}(fg).
\]
A corollary to that is the domain inclusion
\[ \dom(\fourier{\Phi}(f) \fourier{\Phi}(g)) \subseteq 
\dom(\fourier{\Phi}(g)) \cap \dom(\fourier{\Phi}(fg)).
\]
For the converse, suppose that $x \in \dom(\fourier{\Phi}(g)) \cap
\dom(\fourier{\Phi}(fg))$ and define $y,z \in X$ by
\[ \fourier{\Phi}(g)x = y\quad  \text{and} \quad \fourier{\Phi}(fg)x = z.
\]
Let $e\in [f]_\calE$ and $e'\in [efg]_\calE$. Then $e'e f\in [g]_\calE$ and  
hence $\Phi(e'efg)x = \Phi(e'ef)y$. Also, $e'e\in [fg]_\calE$ and hence 
$\Phi(e'efg)x = \Phi(e'e)z$. It follows that 
\[ \Phi(e')\Phi(ef)y = \Phi(e'ef)y = \Phi(e'efg)x = \Phi(e'e)z=
\Phi(e') \Phi(e)z
\]
Since $[efg]_\calE$ is an anchor set, $\Phi(ef)y = \Phi(e)z$.
Since $e\in [f]_\calE$ was arbitrary, $\fourier{\Phi}(f)y = z$, and 
hence $x\in \dom(\fourier{\Phi}(f) \fourier{\Phi}(g))$.

\prfnoi
(FC4) is satisfied by construction. This concludes the proof.
\end{proof}

\medskip

\subsection{The Maximal Anchored Subalgebra}

In practice, $\calF$ may be too large and may fail to satisfy the
anchor-condition 4).
In this case one might look for the maximal subalgebra of $\calF$ which is anchored in $\calE$. 
However, it is not obvious that such an object exists. 
To see that it does,  let us define
\beq\label{ext.eq.ancgen-def} 
 \ancgen{\calE,\calF,\Phi} := \{ f\in \calF \suchthat \forall\,
e\in \calE : [ef]_\calE
\,\,\text{is a $\Phi$-anchor set}\}.
\eeq

\begin{lem}\label{ext.l.non-deg}
Let $\calF$ be a unital algebra, $\calE\subseteq \calF$ a 
subalgebra and $\Phi: \calE \to \BL(X)$ an algebra representation.
Then the following statements are equivalent:
\begin{aufzii}
\item $\calE$ is an anchor set.
\item $\car$ is anchored in $\calE$.
\item $\calE \neq \leer$ and $\ancgen{\calE,\calF,\Phi} \neq \leer$. 
\end{aufzii}
\end{lem}

\begin{proof}
Straightfoward.
\end{proof}

The algebra representation $\Phi: \calE \to \BL(X)$ is called {\emdf
  non-degenerate}
if (i)--(iii) from Lemma \ref{ext.l.non-deg} are satisfied, otherwise {\emdf degenerate}.

\begin{rem}\label{ext.r.degenerate}
If $\Phi$ is degenerate then there are two possibilities: 1st case: 
$\car \notin \calE$. Then   $\calE ':= \calE \oplus \C \car$ is a unital subalgebra of $\calF$
and by
\[ \fourier{\Phi}(f) := \Phi(e) + \lambda \Id,\qquad f = e +
\lambda\car,\, e\in \calE,\,\lambda \in \C
\]
a unital representation $\fourier{\Phi}: \calE \oplus \C\car \to \BL(X)$ is
defined. This new representation is clearly non-degenerate if $X \neq \{0\}$.
2nd case:  $\car \in \calE$. Then $P := \Phi(\car)$ is a projection
and one
can restrict the representation to $\BL(Y)$, where $Y:= \ran(P)$.  

All in all we see that degenerate representations can be neglected.
\end{rem}

\begin{thm}\label{ext.t.ancgen}
Let $\calF$ be a unital algebra, $\calE \subseteq \calF$ a subalgebra
and $\Phi: \calE \to \BL(X)$ a non-degenerate algebra
representation.  Then $\ancgen{\calE, \calF, \Phi}$ is
a unital subalgebra of $\calF$ containing $\calE$ and anchored in $\calE$. 
Moreover, $\ancgen{\calE, \calF, \Phi}$ contains each unital subalgebra of $\calF$  with these properties.
\end{thm}

\begin{proof}
For the proof we abbreviate $\calF' := \ancgen{\calE,\calF,\Phi}$.

\prfnoi
Suppose that $\calF_0$ is a unital subalgebra of $\calF$ that contains
$\calE$ and is anchored in $\calE$. If $f\in \calF_0$ and $e\in \calE$
then  $ef\in \calF_0$ again and hence $[ef]_\calE$ is an anchor set.
This shows that $\calF_0 \subseteq \calF'$.

\prfnoi
As $\Phi$ is non-degenerate, $\calE \subseteq \calF'$ and $\car \in
\calF'$. Let $f\in \calF'$. Then 
\[  \bigcup_{e\in \calE} [ef]_\calE e \subseteq [f]_\calE.
\]
For each $e\in \calE$, $[ef]_\calE$ is an anchor set (since $f\in
\calF'$) and $\calE$ is an anchor set (since $\Phi$ is
non-degenerate). It follows that $[f]_\calE$ is an anchor set as
well. Since $f\in \calF'$ was arbitrary, $\calF'$ is anchored in $\calE$.

\prfnoi
It remains to show that $\calF'$ is a subalgebra of $\calF$. To this
end, fix $f,g\in \calF'$. Then 
\[  \bigcup_{e\in [f]_\calE} [eg]_\calE e \subseteq [f]_\calE \cap
[g]_\calE \subseteq [f+g]_\calE.
\]
It follows that $[f+g]_{\calE}$ is an anchor set.  Since by definition
$\calE \cdot \calF' \subseteq \calF'$, it follows that 
$[d(f+g)]_\calE = [df  + dg]_\calE$ is an anchor set  for each $d\in
\calE$. Hence, $f+g \in \calF'$. 

Likewise, the inclusion 
\[     \bigcup_{e\in [f]_\calE} [efg]_\calE e \subseteq [fg]_\calE
\]  
implies that $[fg]_\calE$ is an anchor set. 
Since as above one can replace here $f$ by $df$ for each $d\in \calE$, it follows that $fg\in \calF'$. 
\end{proof}

\begin{rem}\label{ext.r.ancgen}
Let, as before,  $\calF$ be a unital algebra, $\calE \subseteq \calF$ a subalgebra
and $\Phi: \calE \to \BL(X)$ a non-degenerate representation.  Then:
\begin{aufzi}
\item {\em $\ancgen{\calE, \calF, \Phi}$ contains each $f\in \calF$
such that $\calZ(\calE) \cap [f]_\calE$ is an anchor set.}

\item  {\em If $\calF$ is commutative, then \quad
$ \ancgen{\calE,\calF,\Phi} 
= \{ f\in \calF \suchthat [f]_\calE \,\,\text{is an anchor set}\}$.}
\end{aufzi}
Indeed, a) follows from the inclusion
\[ \calZ(\calE) \cap [f]_\calE \subseteq [ef]_\calE 
\quad \text{for all $e\in \calE$},
\]
which is easy to establish. And b) follows from a).
This shows that our present approach generalizes the
one in \cite[Chapter 7]{HaaseLFC}. 
\end{rem}

\vanish{
This remark is relevant, e.g. for the situation when
$\calF$ consists of operator-valued functions defined on some subset
of $\C$ and $\calE$ is characterized by a pure growth condition. Then
$\calZ(\calE)$ contains scalar-valued functions which may suffice to
regularize an operator-valued function $f$ multiplicatively.}

Let us summarize the results of this section by
combining Theorems \ref{ext.t.ext} and \ref{ext.t.ancgen}.

\begin{cor}\label{ext.c.ancgen}
Let $\calF$ be a unital algebra, $\calE \subseteq \calF$ a subalgebra
and $\Phi: \calE \to \BL(X)$ a non-degenerate representation. Then 
there is a unique  extension $\fourier{\Phi}$ of $\Phi$ to
a calculus on $\ancgen{\calE, \calF, \Phi}$. Moreover,
$\calE$ is an algebraic core for $\fourier{\Phi}$.  
\end{cor}

Corollary \ref{ext.c.ancgen} allows to extend any non-degenerate representation
$\Phi$ of a  subalgebra $\calE$ of a unital algebra $\calF$ to
the subalgebra $\ancgen{\calE,\calF,\Phi}$ of $\calE$-anchored elements. We shall call this the
{\emdf canonical extension} of $\Phi$ within $\calF$, and denote it again by $\Phi$
(instead of $\fourier{\Phi}$ as in the corollary).

\medskip

\subsection{Successive Extensions}\label{ext.s.succ}

Very often, one performs an algebraic extension in a situation,
when there is already some calculus present.
The following
situation is most common:

Let $\calF$ be a unital subalgebra of a unital algebra $\calG$,
and let $\calE \subseteq \calF$ be a subalgebra which is
an  algebraic core for a calculus $\Phi: \calF \to \Clo(X)$. 
Furthermore, let $\calE'$ be a subalgebra of $\calG$ and
$\Psi: \calE' \to \BL(X)$ a representation
with
\[ \calE \subseteq \calE', \qquad \Psi\res{\calE} = \Phi\res{\calE}.
\]
Then $\Psi$ is non-degenerate, and one can perform
an algebraic extension within $\calG$, yielding
\[ \calF':= \ancgen{\calE',\calG, \Psi}.
\]
We denote the extension again by $\Psi$.
The following picture illustrates the situation%
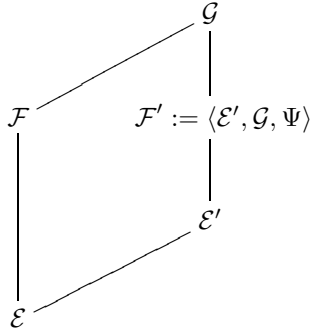
\footnote{Observe that since $\calE$ is an algebraic core
for $\Phi$, the calculus on $\calF$ can be considered
an algebraic extension of $\Phi\res{\calE}$. Hence
the title ``Successive Extensions''.}:
\[
\xymatrix{
&  \calG \\
\calF \ar@{-}[ur]&  \quad  \calF':= \ancgen{\calE',\calG, \Psi} \ar@{-}[u]\\
& \calE' \ar@{-}[u]\\
\calE \ar@{-}[ur]\ar@{-}[uu]&  
}
\]
For a function $f\in \calF$ one may ask, under which conditions
one has $f\in \calF'$ and $\Psi(f) = \Phi(f)$. 
The following result gives some answer.

\begin{thm}\label{ext.t.succ-comp}
In the situation described above, let $f\in \calF$.
Then
\[ f\in \calF' \quad \text{and}\quad\Phi(f) = \Psi(f)
\]
if any one of the following conditions is satisfied:
\begin{aufziii}
\item $f\in \calF'$ and $\Phi'(f)\in \BL(X)$.
\item $f\in \calF'$ and $\Phi'(f)$ is densely defined and
$\Phi(f) \in \BL(X)$.
\item For each $e'\in \calE'$ there is a $\Psi$-anchor
set $\calM_{e'}\subseteq \calE'$ such that 
$\calM_{e'} e'\subseteq \calD'\cdot [f]_\calE$, where
\[ \calD' := \{ d'\in \calF' \suchthat d'\cdot\calE \subseteq \calE'\}.
\]
\item $\calE'= \calE$.
\item $\cent(\calE') \cap [f]_\calE$ is an anchor set.
\item $\calE \subseteq \cnt(\calE')$.
\item $\calE'$ is commutative.
\end{aufziii}
\end{thm}

\begin{proof}
1) and 2): If  $f \in  \calF \cap \calF'$  then $\Phi'(f)
\subseteq \Phi(f)$ by Lemma \ref{urc.l.uni}. Then 1) is sufficient
since $\Phi(f)$ is an operator, and 2) is sufficient since 
$\Phi'(f)$ is closed.

\prfnoi
3)\  We prove first that $f\in \calF'$. Let
$e'\in \calE'$. Take $\calM_e'$ as in the hypotheses. Then
\[ \calM_{e'}e'f \subseteq \calD' [f]_\calE f \subseteq
\calD'\calE \subseteq \calE'.
\]
It follows that $\calM_{e'} \subseteq [e'f]_{\calE'}$. Since $e'\in
\calE'$ was arbitrary, $f\in \calF'$. 
(Recall that 
$\calF' =\ancgen{\calE',\calG,\Psi}$ and cf. 
\eqref{ext.eq.ancgen-def}). 

\prfnoi
For the identity  $\Phi(f) = \Psi(f)$
it suffices to show that $[f]_\calE$ determines $\Psi(f)$.
But this follows directly from Proposition \ref{anc.p.EM},
part 2), with $(\Phi, \calF)$ replaced by $(\Phi', \calF')$
and $\calM := \calE'$.

\prfnoi
Let us now examine the cases 4)--7). 
In case 4), one has $\calE = \calE'$ and one can take
$\calM_{e'} = [e'f]_\calE$ for $e'\in \calE$ in 3).
In case 5) one can take 
$\calM_{e'} = \cnt(\calE') \cap [f]_\calE$ independently of $e'\in
\calE'$. Case 6) is an instance of case 5), since
$[f]_\calE$ is an anchor set by the assumption that
$\calE$ is an algebraic core for $\Phi$ on $\calF$. Finally,
case 7) obviously implies case 6). 
\vanish{

\prfnoi
Next, note that $\Psi(f) \subseteq \Phi(f)$ by 
Lemma \ref{urc.l.uni} (applied to $\calF \cap \calF'$). To prove the
converse inclusion, take $x,y\in X$ with $\Phi(f)x = y$ and 
$e'\in [f]_{\calE'}$. We need to show that $\Psi(e'f)x= \Psi(e')y$.

To this end, fix $c' \in \calM_{e'}$ as in the hypothesis. 
By assumption, there is $d'\in\calD'$ and $e\in [f]_\calE$ such that 
$c'e'= d'e$. Since $\Phi(f)x= y$ we obtain $\Phi(ef)x= \Phi(e)y$ and hence 
\begin{align*} \Psi(c')\Phi(e'f)x & = 
\Psi( c'e'f)x = \Psi(d'ef)x
= \Psi(d') \Psi(ef)x = 
\Psi(d') \Phi(ef)x 
\\ & = 
\Psi(d') \Phi(e)y = 
\Psi(d') \Psi(e)y
= \Psi(d'e)y = \Psi(c'e ')y 
\\ & = \Psi(c')\Psi(e')y.
\end{align*}
Since, by assumption, $\calM_{e'}$ is an anchor set for $\Psi$,
it follows that $\Psi(e'f)x = \Psi(e')y$ as desired.
} 
\end{proof}

\section{Approximate Identities}\label{s.api}

Let $\calF$ be a {\em commutative} unital algebra, 
$\Phi: \calF \to \Clo(X)$ a proto-calculus, and 
 $\calE \subseteq \bdd(\calF,\Phi)$ a  subset of
 $\Phi$-bounded elements.  A sequence $(e_n)_n$  in $\calE$ is called
 a  {\emdf (weak) approximate identity} in $\calE$
(with respect to $\Phi$), if $\Phi(e_n) \to \Id$ strongly (weakly)
as $n \to \infty$.

\medskip

Let $f\in \calF$.
A (weak) approximate identity $(e_n)_n$ is said to be
a (weak) approximate identity {\emdf for $f$}, if 
\[  \Phi(e_n)\Phi(f)  \subseteq \Phi(f_n)\Phi(e_n) 
= \Phi(f_ne_n) \in \BL(X)
\quad \text{for all $n \in \N$}.
\]
More generally, $(e_n)_n$ is said to be
a {\emdf common} 
(weak) approximate identity {\em for} all the elements
of subset $\calM \subseteq \calF$, if $(e_n)_n$ 
is a (weak) approximate identity for each $f\in \calM$.

Finally,  we say that 
$f\in \calF$ {\emdf admits} a (weak) approximate identity
in $\calE$ if there  is a (weak) approximate identity
for $f$ in $\calE$. 
More generally, we say that the elements of a subset $\calM
\subseteq \calF$ {\emdf admit} a common (weak) approximate identity in
$\calE$, if there is a common
(weak) approximate identity in $\calE$ for them. 

\medskip
 
Note that by the uniform boundedness principle, a weak approximate identity $(e_n)_n$ is uniformly $\Phi$-bounded, i.e., satisfies
$\sup_{n\in \N} \norm{\Phi(e_n)} < \infty$.

\begin{lem}\label{api.l.api}
Let $(e_n)_n$ be a weak approximate identity for $f\in \calF$
with respect to $\Phi$ and let 
\[ D := \spann\bigcup_{n\in \N} \ran(\Phi(e_n)).
\]
Then the following assertions hold:
\begin{aufzi}
\item $\{e_n \suchthat n \in \N\}$ is an anchor set.
\item $D$ is dense in $X$  and $D \subseteq \dom(\Phi(f))$.
In particular, $\Phi(f)$ is densely defined. 
\item $\dom(\Phi(f)) \cap D$ is a core for $\Phi(f)$. 
If $(e_n)_n$ is an aproximate identity for $f$ then
$\Phi(e_n)x \to x$ within the Banach space $\dom(\Phi(f))$
for each $x\in \dom(\Phi(f_n))$.
\item For all $n\in \N$
\[  \cls{\Phi(e_n)\Phi(f)} = \Phi(e_nf) = \Phi(fe_n)
= \Phi(f)\Phi(e_n).
\] 
\end{aufzi}
\end{lem}

\begin{proof}
a) is trivial and b) follows from Mazur's theorem, as
$D$ is clearly weakly dense in $X$. 

\prfnoi
c)\ Let $x,y\in X$ with $\Phi(f)x = y$. Then by hypothesis, 
for each $n \in \N$ we have $(\Phi(e_n)x, \Phi(e_n)y) \in \Phi(f)$, so
$\Phi(e_n)x\in D \cap \dom(\Phi(f))$. 
Since $(\Phi(e_n)x, \Phi(e_n)y) \to (x,y)$ weakly, the space
$\Phi(f)\res{D}$---considered as a
subspace of $X\oplus X$---is weakly dense in  $\Phi(f)$. By Mazur's theorem again, this space is
strongly dense, hence $D$ is a core for $\Phi(f)$.
If $(e_n)_n$ is even a strong approximate identity, then
$\Phi(e_n)x \to x$ and $\Phi(f)\Phi(e_n)x = \Phi(e_n)y 
\to y$ strongly.

\prfnoi
d) Suppose that $\Phi(fe_n) \in \BL(X)$ for all $n\in \N$. 
Then  $\Phi(f) \Phi(e_n) = \Phi(e_nf) \in \BL(X)$, and hence
$\ran(\Phi(e_n)) \subseteq \dom(\Phi(f))$. If follows that 
$D\subseteq \dom(\Phi(f))$, and $\Phi(f)$ is densely defined,
by b). By hypothesis,
\[ \Phi(e_n)\Phi(f) \subseteq \Phi(f) \Phi(e_n) = \Phi(fe_n).
\]
Since the left-most operator is densely defined, 
we obtain
\[ \cls{\Phi(e_n)\Phi(f)} = \Phi(fe_n).
\]
On the other hand, by (FC3), 
\[ \Phi(e_n)\Phi(f) \subseteq \Phi(e_nf)
\]
and the latter is a closed operator. It follows that 
$\Phi(e_nf) = \Phi(fe_n)$ as claimed. 
\end{proof}

By virtue of the preceding lemma, we obtain the following
result.

\begin{thm}\label{api.t.api}
Let $\Phi: \calF \to \BL(X)$ be a proto-calculus. 
\begin{aufzi}
\item
If  $(e_n)_n$ is a (weak) approximate identity
for $f,g\in \calF$, then it is a (weak) approximate identity for
$f+g$ and $\lambda f$ ($\lambda \in \C$) and one has
\[ \cls{\Phi(f) + \Phi(g)} = \Phi(f + g).
\]
\item 
If  $(e_n)_n$ is a strong approximate identity
for $f,g\in \calF$, then $(e_n^2)_n$ is a strong 
approximate identity for $fg$, and one has
\[ \cls{\Phi(f) \Phi(g)} = \Phi(fg).
\]
\end{aufzi}
\end{thm}

\begin{proof}
a)\ Since $\Phi(e_nf) = \Phi(fe_n)$ and $\Phi(e_ng) = \Phi(ge_n)$
are bounded, so is 
\[ \Phi(e_n(f+g)) = \Phi(e_n f + e_n g) = \Phi(e_nf) + \Phi(e_n g)
= \dots = \Phi((f+g)e_n).
\]
It follows that $(e_n)_n$ is a (weak) approximation of identity for
$f+g$ and, hence, that $D$ is a core for $\Phi(f+g)$. But 
$D \subseteq \dom(\Phi(f)) \cap \dom(\Phi(g))$, and so we are done.

\prfnoi
b) Since $(e_n)_n$ is an approximate identity, it is bounded, and
hence also $(e_n^2)_n$ is an approximate identity.  Note that
\begin{align*}
 \Phi(fge_n^2) & = \Phi(f(ge_n)e_n)
= \Phi(f) \Phi(ge_n) \Phi(e_n)
= \Phi(f) \Phi(e_ng) \Phi(e_n)
\\ & = \Phi(fe_nge_n) = \Phi(fe_n) \Phi(ge_n) \in \BL(X),
\end{align*}
and continuing the computation yields
$\Phi(fge_n^2) = \Phi(e_n^2fg)$.  This proves the first claim. 
The second follows easily.
\end{proof}

\vanish{
\begin{exa}
If $\Phi: \Meas(K,\Sigma) \to \BL(H)$ is any measurable functional
calculus then each $f\in \Meas(K,\Sigma)$ admits an approximate
identity in $\calE = \BMb(K,\Sigma)$,  for instance $e_n := \frac{n}{n
  + \abs{f}}$ or $e_n := \car_{\set{\abs{f}\le n}}$,  $n\in \N$. 
\end{exa}

\begin{exa}
Suppose that $-A$ generates a bounded $C_0$-semigroup on a Banach
space $X$ and  $f\in \Mer(\C_+)$ is such that $f(A)$ is defined 
in the extended Hille--Phillips calculus  for $A$. If
\[ D_\infty := \bigcap_{n\ge 0} \dom(A^n)
\]
is contained in $\dom(f(A))$ then $f$ admits an approximate identity. 
In particular, $D_\infty$ is a core for $f(A)$ (Exercise \ref{mafc.ex.D-infty}).
\end{exa}

}


\section{The Dual Calculus}\label{s.dua}

For a calculus   $(\calF, \Phi)$ 
on  a Banach space $X$ one is tempted
to define a ``dual
calculus'' on $X'$ by letting $\Phi'(f) := \Phi(f)'$.  This is
premature in at least two respects. First, if $\Phi(f)$ is not
densely defined, $\Phi(f)'$ is just a linear relation and not an operator.  Secondly, 
even if the first problem is ruled out by appropriate minimal 
assumptions,
it is not clear how to establish the formal properties of a
calculus for $\Phi'$. 

\medskip

To tackle these problems, we shall take a different route and define the dual calculus by virtue of the extension
procedure described in Section \ref{s.ext}. 
To wit, let
$\Phi:\calF \to \Clo(X)$ be any proto-calculus and let  $\calB :=
\bdd(\calF, \Phi)$ the set of $\Phi$-bounded elements. For
$b \in \calB$ we define
\[  \Phi'(b) := \Phi(b)' \in \BL(X').
\]
As $\calF$ may not be commutative, the mapping
$\Phi'$ may not be a homomorphism for the original algebra structure. 
To remedy this defect, we pass to the {\emdf opposite algebra}
$\calF^\op$, defined on the same set $\calF$ 
with the same linear structure but
with the ``opposite''  multiplication 
\[ f \cdot_\op g = gf\qquad (g,f\in \calF).
\]
For any subset $\calM \subseteq \calF$ we write $\calM^\op$ when we
want to consider $\calM$ as endowed with this new multiplication. 
This applies in particular to $\calB$, whence we obtain
\[ [f]_{\calB^\op} = \{ e\in \calB \suchthat e\cdot_\op f \in \calB\}
= \{ e\in \calB \suchthat fe \in \calB\}
\]
for $f\in \calF^\op$.   The mapping
\[ \Phi': \calB^\op \to \BL(X')
\]
is a unital algebra homomorphism.  We then can pass to
its canonical extension to the algebra
\[ \calF' := \ancgen{\calB^\op, \calF^\op, \Phi'};
\]
as usual, we shall denote that extension by $\Phi'$ again.  The  mapping
\[ \Phi': \calF' \to \BL(X')
\]
is called the {\emdf dual calculus} associated with  $\Phi$. 
By construction, it is a calculus (Theorem \ref{ext.t.ext}).

\begin{thm}\label{dua.t.main}
Let $(\calF, \Phi)$ be a proto-calculus  with dual calculus
$(\calF', \Phi')$, and let $f\in \calF'$. Define
\[ D_f := \spann\{ \Phi(e)x \suchthat x\in X, \,\,e\in
[f]_{\calB^\op}\}.
\]
Then the following assertions hold: 
\begin{aufzi}
\item $D$ is dense in $X$ and $D \subseteq \dom(\Phi(f))$. 
In particular, $\Phi(f)$ is densely defined.
\item $\Phi(f)'\subseteq \Phi'(f)$, with equality if and only if
$D_f$ is a core for $\Phi(f)$. 
\item $\Phi(f)$ is bounded if and only if $\Phi'(f)$ is bounded;  in this case $\Phi'(f)= \Phi(f)'$.
\end{aufzi} 
\end{thm}

\begin{proof}
a) Let $e \in [f]_{\calB^\op}$. Then $e, fe\in \calB$. Hence,
$\ran(\Phi(e)) \subseteq \dom(\Phi(f))$. This yields the inclusion
$D \subseteq \dom(\Phi(f))$.  Since by construction,
$[f]_{\calB^\op}$ is a $\Phi'$-anchor, one has
\beq 
\bigcap_{e\in [f]_{\calB^\op}} \ker(\Phi(e)') = \{0\}.  
\eeq
By a standard application of the Hahn--Banach theorem, $D$ is dense in $X$.

\vanish{
\prfnoi
b) Let $e \in [f]_{\calB^\op}$ as before.
Then, 
\[ 
\Phi'(e\cdot_\op f) = \Phi(fe)'= (\Phi(f)
\Phi(e))' \supseteq  \Phi(e)'\Phi(f)'= \Phi'(e)\Phi(f)'. 
\]
It follows that if $x',y'\in X'$ are such that $\Phi(f)'x' = y'$, then 
\[ \Phi'(e\cdot_\op f)x'= \Phi'(e)y'\]
for all $e\in [f]_{\calB^\op}$, and hence $\Phi'(f)x' = y'$ by
construction  of the dual calculus. 
}

\prfnoi
b)\ Fix $x',y'\in X$. Note the following equivalences:
\begin{align*}
\Phi'(f)x'= y' & \gdw\,\, \forall e\in [f]_{\calB^\op} :
                 \Phi'(e\cdot_\op f)x'=
                 \Phi'(e) y'
\\ & \gdw\,\, \forall e\in [f]_{\calB^\op} : \Phi(fe)'x'=
                 \Phi(e)' y'
\\ & \gdw\,\, \forall e\in [f]_{\calB^\op},\, z\in X : \dprod{\Phi(f)\Phi(e)z}{x'} =
     \dprod{\Phi(e)z}{y'}
\\ & \gdw\,\, (x',-y') \perp (\Phi(f) \cap (D_f \oplus X)),
\end{align*}
where we identify $\Phi(f)$ with its graph as a subset of $X \oplus
X$.  On the other hand,
\[ \Phi(f)'x'= y ' \,\,\gdw\,\, (x',-y') \perp \Phi(f).
\]
From this it is evident that $\Phi(f)'\subseteq
\Phi'(f)$. Furthermore, since both operators $\Phi(f)'$ and $\Phi'(f)$
are weakly$^*$ closed, by the Hahn--Banach theorem one has equality
if and only if $\Phi(f) \cap (D_f \oplus X))$ is dense
in $\Phi(f)$. The latter just means that $D_f$ is a core for $\Phi(f)$.

\prfnoi
c)\ If $\Phi(f)$ is bounded, then so is $\Phi(f)'$, and hence by b)
$\Phi'(f) = \Phi(f)'$. Suppose, conversely, that $\Phi'(f) \in
\BL(X')$.  Since  $\Phi'(f)$ has a closed graph for the weak$^*$
topology,  it follows from Theorem \ref{sal.t.cgt} that there is $T\in \BL(X)$ such that 
$\Phi'(f) = T'$.  By b), $\Phi(f)'\subseteq T'$, which in turn implies
that 
\[ T = T''\cap (X\oplus X) \subseteq \Phi(f)'' \cap (X\oplus X) 
= \Phi(f), 
\]
since $\Phi(f)$ is closed, see \cite[Prop.A.4.2.d]{HaaseFC}. 
This implies that $\Phi(f) = T$, so $\Phi(f)$ is indeed bounded.
\end{proof}

A calculus $(\calF, \Phi)$ on a Banach space $X$ is called {\emdf
  dualizable} if  $\calF'= \calF^\op$, i.e., the dual calculus
is defined on $\calF^\op$. Equivalently, $(\calF, \Phi)$ is dualizable
if for each $f\in \calF$ and each $b \in \bdd(\calF, \Phi)$ the space
\[ D_{fb} = \spann\{ \Phi(e)x \suchthat x\in X,\, e, fbe \in
\bdd(\calF, \Phi)\}
\]
is dense in $X$. For a dualizable calculus one has
\[ \bdd(\calF, \Phi) = \bdd(\calF', \Phi')
\]
by c) of Theorem \ref{dua.t.main}.

\medskip
If  $\Phi$ on $\calF$ is a non-dualizable calculus then
${\calF'}^\op$ (i.e., $\calF'$ with the original algebra structure) is
a $\Phi$-regular subalgebra of $\calF$ (since it contains
$\bdd(\calF, \Phi)$)
and we may restrict $\Phi$ to
this algebra. In a sense, ${\calF'}^\op$ is the largest
subalgebra such that the restriction
of $\Phi$ to it is a dualizable calculus.

\section{Topological Extensions}\label{s.top}

The algebraic extension procedure discussed in Section \ref{s.ext}
is based on a ``primary'' or ``elementary'' calculus $\Phi: \calE \to
\BL(X)$ that can be extended. In this section we discuss the form
of possible other----topological---ways of extending a primary
calculus. Whereas the algebraic extension is canonical when a
superalgebra is given, a topological extension depends also on the
presence of a given topological structure on the superalgebra. 

In the following we want to formalize the idea of a topological 
extension in such generality that the extant examples are covered. 
However, we admit that experience with topological extensions as such
is scarce, so that the exposition given here is 
likely to be replaced by a better one some time in the future.

\medskip
Let $\calF$ be an algebra and $\Lambda$ a set.
An {\emdf (algebraic) convergence
structure} on $\calF$ over $\Lambda$ is a relation 
\[  \tau \subseteq \calF^\Lambda\times  \calF
\]
with the following properties: 
\begin{aufziii}
\item $\tau$ is a subalgebra of $\calF^\Lambda \times \calF$. 
\item For each $f\in \calF$ the pair $((f)_{\lambda \in \Lambda}, f)$ is in $\tau$.
(Here, $(f)_{\lambda \in \Lambda}$ is the constant family.) 
\end{aufziii}
The convergence structure is called {\emdf Hausdorff}, if
$\tau$ is actually an operator and not just a relation. If $\Lambda =
\N$, we speak of a {\emdf sequential}  convergence structure.

Given a convergence structure $\tau$, 
one writes $f_\lambda \stackrel{\tau}{\to} f$ in place of 
$((f_\lambda)_\lambda, f) \in \tau$ and says that $(f_\lambda)_\lambda$
{\emdf $\tau$-converges to} $f$. 
From 1) and 2) it follows that 
$\dom(\tau) \subseteq \calF^\Lambda$ is an algebra containing all
constant families. The structure $\tau$ is Hausdorff if and only
if one has 
\[ f_\lambda \stackrel{\tau}{\to} f,\, f_\lambda \stackrel{\tau}{\to} g
\quad \dann\quad f=g.
\]

\medskip
From now on, we consider the following situation: 
$\calE'$ is a unital algebra, $\calE \subseteq \calE'$ is a subalgebra,
and $\Phi: \calE \to \BL(X)$ is a representation; 
$\calA \subseteq \BL(X)$ is a subalgebra such that
$\Phi(\calE)\subseteq \calA$; and 
$\tau = (\tau_1, \tau_2)$ is a pair of 
convergence structures $\tau_1$  on $\calE'$ and $\tau_2$ on $\calA$ 
over the same index set $\Lambda$.
(The latter will be called a {\emdf joint convergence structure} 
on $(\calE',\calA)$
in the following.)

In this situation, the set
\[ \calE^{\tau} := \{ f\in \calE' \suchthat \exists\, (e_\lambda)_\lambda \, \text{in
  $\calE$},\, T\in \calA : e_\lambda \stackrel{\tau_1}{\to} f,\,\,
\Phi(e_\lambda) 
\stackrel{\tau_2} \to
T\}
\]
is a subalgebra of $\calE'$ containing $\calE$. 
Suppose in addition that $\Phi$ is {\emdf closable with respect to
  $\tau$},
which means that 
\beq\label{sec.eq.tau-closable}
   (e_\lambda)_\lambda \in \calE^\Lambda,\,T\in \calA,\,  
  e_\lambda \stackrel{\tau_1}{\to} 0,\, \Phi(e_\lambda) 
   \stackrel{\tau_2}{\to} T\, \quad \dann\quad T=0.
\eeq
Then one can define the {\emdf $\tau$-extension} $\Phi^\tau :
\calE^\tau \to \BL(X)$  of $\Phi$  by
\[ \Phi^\tau(f) := T
\]
whenever $(e_\lambda)_\lambda \in \calE^\Lambda$, $e_\lambda \stackrel{\tau_1}{\to}f$, and
$\Phi(e_\lambda) \stackrel{\tau_2}{\to} T$.
(Indeed, \eqref{sec.eq.tau-closable} just guarantees that $\Phi^\tau$ is
well-defined, i.e., $\Phi^\tau(f)$ does not depend on the chosen $\tau$-approximating
sequence $(e_\lambda)_\lambda$.)

\begin{thm}\label{top.t.top-ext}
The so-defined mapping $\Phi^\tau: \calE^\tau \to \BL(X)$ is 
an algebra homomorphism which extends $\Phi$.
\end{thm}

\begin{proof}
Straightfoward.
\end{proof}

In practice,  one wants to combine a topological with 
an algebraic extension, and that raises a compatibility
issue. To explain this, let us be more specific.

\medskip

\noindent
Let $\calE$ be an algebraic core for
a calculus $\Phi: \calF \to \Clo(X)$, let 
$\calG$ be a superalgebra of $\calF$ and let
$\calE'$ be a subalgebra of $\calG$ containing $\calE$:
\[ \calF \subseteq \calG\quad \text{and}\quad
\calE \subseteq \calE' \subseteq \calG.
\] 
Suppose further that $\tau= (\tau_1, \tau_2)$ is a joint convergence
structure on $(\calE', \calA)$, where $\calA$ is a subalgebra of
$\BL(X)$ containing $\Phi(\calE)$, and that $\Phi\res{\calE}$ is 
closable with respect to $\tau$. 

As above,  let $\Phi^\tau$ denote the $\tau$-extension of
$\Phi\res{\calE}$ to the
algebra $\calE^\tau \subseteq \calE'$. Starting from
$\calE^\tau$ we can extend $\Phi^\tau$ algebraically to
\[ \calG^\tau := \ancgen{\calE^\tau, \calG, \Phi^\tau},
\]
and we denote this extension again by $\Phi^\tau$.

The question arises whether $\Phi^\tau$ is an extension
of $\Phi$. This problem has been already discussed in
Section \ref{ext.s.succ} in a more general context, so that  
Theorem \ref{ext.t.succ-comp} and the subsequent remarks apply.  
In particular, we obtain the following:

\begin{cor}\label{top.c.top-ext-comp}
In the situation described above, the following assertions hold:
\begin{aufzi}
\item If $f\in \calF \cap \calG^\tau$ then $\Phi^\tau(f) \subseteq
\Phi(f)$, so that $\Phi^\tau(f) = \Phi(f)$ if $\Phi^\tau(f)$ is
bounded. In particular, $\Phi^\tau= \Phi$ on $\calE^\tau \cap
\calF$. 
\item If $f\in \calF$ is such that 
$\cnt(\calE^\tau) \cap [f]_\calE$ is an anchor set, then
$f\in \calG^\tau$ and $\Phi^\tau(f) = \Phi(f)$. 
In particular, $\Phi^\tau$ extends $\Phi$ if
$\calE \subseteq \cnt(\calE^\tau)$.
\end{aufzi}
\end{cor}

If $\calE$ is commutative and $\tau_1$ is Hausdorff, then
$\calE^\tau$ is also commutative. Hence, in this case, 2)
is applicable and it follows that $\Phi^\tau$ extends
$\Phi$.


\begin{rems}
The idea of a topological extension in the abstract theory of
  functional
calculus was introduced in \cite{Haa05b}, with  a (however
easy-to-spot)   mistake in the formulation  of closability \cite[(5.2)]{Haa05b}.
There, with an immediate application in mind, the discussion was
still informal. 

In our attempt to formalize it, we here introduce
the notion of a ``convergence structure'' which, admittedly, 
  is {\em ad hoc}.  We have not browsed through the literature to find a suitable
and already established notion. Maybe one would prefer a richer axiomatic
tableau for  such a notion and is tempted to add axioms, e.g., 
require that $\Lambda$ is directed and that families that coincide
eventually display the same convergence behaviour.
On the other hand, only axioms
1) and 2) are needed to prove Theorem \ref{top.t.top-ext}. 

In practice, one may often  choose $\tau_2$ to be operator norm or strong
convergence, but other choices are possible. (See Sections
\ref{s.sectop} and \ref{sgr.s.topext} below for
an interesting example. It was the latter example that led us to acknowledge
that one needs flexibility of the convergence notion also on the
operator side.)
\end{rems}

\part{Examples}

In this second part of the article,  
we want to illustrate the abstract theory with
some well-known examples. However, we focus on the supposedly  less
well-known aspects.


\section{Sectorial Operators}\label{s.sec}

A closed operator $A$ on a Banach space $X$ is called {\emdf
sectorial} if there is $\omega \in [0, \upi)$ such that 
$\spec(A)$ is contained in the sector $\cls{\sector{\omega}}$
and the function $\lambda \mapsto \lambda R(\lambda,A)$ is
uniformly bounded outside every larger sector.  The minimal
$\omega$ with  this property is called the
{\emdf sectoriality angle} and is denoted by $\omega_\sct(A)$.

For $\omega >  0$ we let
$\calE(\sector{\omega})$ be the set of functions
$f\in \Ha^\infty(\sector{\omega})$ such that 
\[  \int_{\rand{\sector{\delta}}} \abs{f(z)}\frac{\abs{\ud{z}}}{\abs{z}}
<\infty \qquad \text{for all $0 \le \delta < \omega$}.
\]
If $f \in \calE(\sector{\omega})$ and 
$A$ is a sectorial operator on $X$ of angle $\omega_\sct(A) < \omega$ then
we define
\beq\label{sec.eq.def} 
\Phi_A^\omega(f) := \frac{1}{2\upi \ui} \int_{\rand{\sector{\delta}}}
f(z) R(z,A) \ud{z}
\eeq
Note that the norm condition on the resolvent of $A$ and the
integrability condition on $f$ just match in order to render
this integral absolutely convergent. It is a classical fact
that 
\[ \Phi_A^\omega: \calE(\sector{\omega}) \to \BL(X)
\]
is an algebra homomorphism and 
\[  \Phi_A^\omega\bigl( \frac{\bfz}{(1 + \bfz)^2}\bigr) = 
A(1+A)^{-1}.
\]
Standard complex analysis arguments yield that for each 
$f\in \calE(\sector{\omega})$ one has
\[ \lim_{z\to 0} f(z) = \lim_{z\to \infty} f(z) = 0
\qquad \text{whenever $\abs{\arg z}\le  \delta$}
\]
for each $0 < \delta < \omega$. As a consequence, 
\[  \ker(A) \subseteq \ker \Phi_A^\omega(f).
\]
It follows that $\Phi_A^\omega$ is degenerate if $A$ is
not injective.  

Since we do not want to assume the  injectivity of $A$,  
we could follow Remark \ref{ext.r.degenerate}
and extend $\Phi^\omega_A$ to the unital algebra
\[ \calE_1(\sector{\omega}) := \calE(\sector{\omega}) \oplus \C \car
\]
and make this the basis of the algebraic extension procedure
from Section \ref{s.ext}. 
However, it is easily seen that the function $(1+\bfz)^{-1}$
is not anchored in $\calE_1(\sector{\omega})$. 
So, the resulting calculus
would be ``too small'' in the sense that it would not cover
some natural functions of $A$. 

In order to deal with this problem, one extends $\Phi_A^\omega$
further to the algebra
\[  \calE_{e}(\sector{\omega}) := 
\calE(\sector{\omega}) \oplus \C\car \oplus \C \frac{1}{1 + \bfz}
\]
by definining
\[ \Phi_A( (1 + \bfz)^{-1}) := (1+A)^{-1}.
\]
It follows from properties of $\Phi_A^\omega$ on 
$\calE(\sector{\omega})$ that this extension is indeed
an algebra homomorphism.

At this point one may perform an algebraic extension as 
in Theorem \ref{ext.t.ext} 
within a ``surrounding'' algebra $\calF$. 
A natural choice for $\calF$ is the 
field $\Mer(\sector{\omega})$ of all meromorphic functions on the sector $\sector{\omega}$.
The domain of the resulting calculus, which
is again denoted by $\Phi_A^\omega$,  is the algebra
\[ \dom(\Phi_A^\omega) := \ancgen{\calE_{e}(\sector{\omega}), \Mer(
\sector{\omega}), \Phi_A^\omega}.
\]
Whereas the algebras $\calE$ and $\calE_e$ are
described independently of $A$, the algebra
$\dom(\Phi_A^\omega)$ is heavily dependent on $A$, and
is, as a whole,  quite arcane in general.

\medskip
Note that there is still a dependence of our calculus on 
the choice of $\omega > \omega_\sct(A)$. This can be
eliminated as follows: for $\omega_1 > \omega_2> \omega_{\sct}(A)$
one has a natural embedding
\[ \eta: \Mer(\sector{\omega_1}) \to 
\Mer(\sector{\omega_2}),\qquad \eta(f) := f\res{\sector{\omega_2}}.
\]
By the identity theorem of complex analysis, 
$\eta$ is injective. Of course we expect
compatibility, i.e, 
\[ \Phi_A^{\omega_2}(f\res{\sector{\omega_2}})
= \Phi_A^{\omega_1}(f)\qquad \text{for each $f\in \dom(\Phi_A^{\omega_2})$}.
\]
Since all involved algebras are commutative,
by  Theorem \ref{urc.t.com} this has to be verified only
for functions $f\in \calE_e(\sector{\omega_1})$, and hence
effectively only for functions $f\in \calE(\sector{\omega_1})$. 
For such functions it is a consequence of a path-deformation
argument.

\medskip
By letting $\omega$ approach $\omega_\sct(A)$ from above,
we obtain a ``tower'' of larger and larger algebras
and  their ``union''
\[  \calE[\sector{\omega_\sct(A)}]
:= \bigcup_{\omega > \omega_\sct(A)} \calE(\sector{\omega}),
\]
and likewise for $\calE_e[\sector{\omega_\sct(A)}]$ and 
$\Mer[\sector{\omega_\sct(A)}]$. A precise definition
of this union would require the notion of a meromorphic function germ
on  $\cls{\sector{\omega_\sct(A)}}\ohne\{0\}$. The resulting
calculus is called the {\emdf sectorial calculus} for $A$
and ist denoted by  $\Phi_A$ here. 
It can be seen as an ``inductive limit'' of the calculi $\Phi_A^\omega$.


\section{Topological Extensions of the Sectorial Calculus}\label{s.sectop}

Of course the question arises whether the sectorial
calculus $\Phi_A$ covers
all 
``natural'' choices for functions $f$ of $A$. The answer
is ``no'', at least when the operator $A$ is not injective.

To understand this, we look at functions of the form
\begin{aufziii}
\item $f(z) = \int_{\R_+} \frac{z}{t+ z} \, \mu(\ud{t})$ 
\quad and
\item $g(z) =  \int_{\R_+} (tz)^n \ue^{-tz}\, \mu(\ud{t})$,
\end{aufziii}
where  $\mu \in \eM(\R_+)$ is a 
complex Borel measure on 
$\R_+ = [0,\infty)$.
It is easy to see that by 1) a holomorphic
function $f$ on $\sector{\upi}$ is defined, bounded on each smaller
sector. And by 2), a holomorphic function $g$ on $\sector{\upi/2}$
is defined, bounded on each smaller sector.

Of course, in any reasonable functional calculus one would expect 
\beq\label{sectop.eq.Hirsch}
  f(A) = \int_{\R_+} A(t+A)^{-1}  \, \mu(\ud{t}) 
\eeq
for each sectorial operator $A$ and 
\beq\label{sectop.eq.sgrp}
 g(A) = \int_{\R_+} (tA)^n\ue^{-tA}  \, \mu(\ud{t})
\eeq
for each sectorial operator $A$ with $\omega_{\sct}(A) < 
\frac{\upi}{2}$.   If $A$ is injective, this is true for
the sectorial calculus.

However, if $A$ is not injective, then 
$\mu$ can be chosen so that $[f]_{\calE_e}$ is
not an anchor set and, consequently, 
$f$ is not contained in the domain of $\Phi_A$.
A prominent example for this situation is the function
\[ f(z) = \frac{1}{\lambda - \log z}
= \int_0^\infty \frac{-1}{(\lambda - \log t)^2 + \upi^2}
(t+ A)^{-1} \, \ud{t},
\]
which plays a prominent role for  Nollau's result on
operator logarithms \cite[Chapter 4]{HaaseFC}.
Similar remarks apply in case 2). 

The mentioned  ``defect'' of the sectorial calculus $\Phi_A$ can be mended by 
passing to a suitable
topological extension as described in Section \ref{s.top}.
Actually, this has been already observed in \cite{Haa05b},
where uniform convergence
was used as the underlying convergence structure.

\medskip
\noindent
Here, we intend to generalize the result from \cite{Haa05b} by 
employing a weaker convergence structure on a larger
algebra.  Define
\[  \Ha^\infty(\sector{\omega}\cup\{0\})
:= \{ f\in \Ha^\infty(\sector{\omega}) \suchthat 
f(0) := \lim_{z\searrow 0} f(z) \,\,\text{exists}\}.
\]
We  say that a
sequence $(f_n)_n$ in $\Ha^\infty(\sector{\omega}\cup\{0\})$
converges {\emdf pointwise and boundedly}
(in short: bp-converges) on $\sector{\omega}\cup\{0\}$
to  $f\in \Ha^\infty(\sector{\omega}\cup\{0\})$
if $f_n(z) \to f(z)$ for each $z\in \sector{\omega}\cup\{0\}$ and
$\sup_n \norm{f_n}_{\infty, \sector{\omega}} < \infty$. It is obvious
that bp-convergence is a Hausdorff sequential convergence structure
as introduced in Section \ref{s.top}.  We take
bp-convergence as the first component of the joint
convergence structure $\tau$ we need for a topological extension.

The second component, is described as follows. 
For a set $\calB \subseteq \BL(X)$ let 
\[ \calB':= \{ S\in \BL(X) \suchthat \forall\, B \in \calB: SB = BS \}
\]
be its {\emdf commutant} within $\BL(X)$.  Let 
\[ \calA_A := \{ (1+A)^{-1}\}'
= \{ R(\lambda, A) \suchthat  \lambda \in \resolv(A)\}'
\]
For $((T_n)_n, T) \in \calA_A^\N
\times \calA_A$ we write
\[ T_n \stackrel{\tau_A^s}{\to} T
\]
if there is a point-separating set $\calD \subseteq \calA_A'$ such that 
\[   DT_n \to DT \quad \text{strongly, for each $D\in \calD$}.
\]
(Recall that $\calF$ is point-separating if $\bigcap_{D\in \calD}
\ker(D) = \{0\}$.)
Note that  $\calA_A'$ is a  commutative unital subalgebra 
of $\BL(X)$ closed with respect to strong convergence and containing
all resolvents of $A$.

\begin{lem}
The relation $\tau_A^s$ on $\calA_A^\N
\times \calA_A$ is an algebraic  Hausdorff convergence structure. 
\end{lem}

\begin{proof}
This follows easily from the fact that $\calA'$ is commutative. 
\end{proof}

\begin{rem}
We note that, in particular, one has
\[  (1 +A)^{-m} T_n \to (1+A)^{-m} T \,\, \text{strongly}
\quad \dann\quad T_n \stackrel{\tau_A^s}{\to} T
\]
that for any $m \in \N_0$.
This means that $\tau_A^s$-convergence is weaker than strong convergence
in any extrapolation norm associated with $A$. 
\end{rem}

We shall show that $\Phi_A$ on $\calE_e(\sector{\omega})$ is
closable with respect to 
the joint convergence structure
\beq  \tau = \text{( bp-convergence on
  $\Ha^\infty(\sector{\omega}\cup\{0\})$ ,
 $\tau_A^s$ on $\calA_A$ )}.
\eeq
We need the following auxiliary
information.



\begin{lem}\label{sectop.l.aux}
Let $A$ be sectorial, let $\omega_\sct(A) < \omega < \upi$ and let 
$e\in \calE(\sector{\omega})$. Then 
\[ \ran(\Phi_A(e)) \subseteq \cls{\ran}(A).
\]
\end{lem}

\begin{proof}
Let $\vphi_n := \frac{n\bfz}{1 + n\bfz}$. 
Then 
$\vphi_n \to \car$ pointwise and boundedly on $\sector{\omega}$. 
By Lebesgue's theorem, $\Phi_A(\vphi_n e)\to \Phi_A(e)$ 
in norm. But
\[  \Phi_A(\vphi_n e) = nA (1 + nA)^{-1} \Phi_A(e). 
\]
The claim follows. 
\end{proof}

Now we can head for the main result.

\begin{thm}\label{sectop.t.main}
Let $A$ be a sectorial operator on a Banach space $X$,
let $\omega\in (\omega_\sct(A),\upi)$ and let 
$(f_n)_n$ be a sequence in $\Ha^\infty(\sector{\omega}\cup\{0\})$  
such that $f_n \to 0$ pointwise and boundedly on $\sector{\omega}
\cup \{0\}$. Suppose that 
$\Phi_A(f_n)$ is defined and  bounded for each $n\in\N$, and
that $\Phi_A(f_n) \stackrel{\tau_2}{\to} T \in \calA_A$ strongly. 
Then $T= 0$.
\end{thm}

\begin{proof}
For simplicity we write $\Phi$ in place of $\Phi_A$, and 
$\calE$ and $\calE_e$ in place of $\calE(\sector{\omega})$ and
$\calE_e(\sector{\omega})$, respectively. 
By passing to $f_n - f_n(0)\car$ we may suppose that 
$f_n(0)= 0$ for each $n\in \N$.

By hypothesis, there is a point-separating set $\calD \subseteq
\calA_A'$ such that
\[   D\Phi(f_n) \to DT \quad \text{strongly, for all $D\in \calD$}.
\]
We fix $D\in \calD$ for the time being.

Now, take  $e := \bfz (1+\bfz)^{-2}$
and observe that $ef_n \in \calE$
and $\Phi(ef_n) \to 0$ in operator norm by Lebesgue's theorem and
the very definition of $\Phi$ in \eqref{sec.eq.def}. 
On the other hand,
\[ D\Phi(ef_n) = D\Phi(e) \Phi(f_n)  = \Phi(e)D \Phi(f_n) \to 
\Phi(e)DT 
\]
strongly.  This yields
$A (1 + A)^{-2}DT = \Phi(e)DT = 0$, and hence
\beq\label{sectop.eq.aux1}   
\ran( (1+A)^{-2}DT) \subseteq \ker(A).
\eeq
If $\ker(A) = \{0\}$ then $DT=0$ and hence $T=0$ since $D$ was
arbitrary from $\calD$.  So suppose that 
$A$ is not injective and define
$e_0 := (1 + \bfz)^{-2}$.

We claim that 
$e_0f_n \in \calE$. To prove this, note
that, by hypothesis, $f_n$ is anchored
in $\calE_e$. Since $A$ is not injective,
$[f]_{\calE_e}$ must contain at least one
function $e_1$ with $e_1(0)\neq 0$. Write 
\[ e_1= e_2 + c\car + \frac{d}{1 + 
\bfz}  = e_2+ \frac{c+d}{1+ \bfz} + c\frac{\bfz}{1 + \bfz}
\quad \text{for certain $c,\:d \in \C$
and $e_2\in \calE$}.
\] 
Now multiply by $f_n$ and $(1 + \bfz)^{-1}$ to obtain
\[  (c+d)e_0f_n = \frac{e_1f_n}{1+ \bfz} - e_2\frac{f_n}{1+ \bfz}
- c \frac{\bfz}{(1+\bfz)^2} \in \calE.
\]
(Note that $e_1 f_n \in \calE_e$.) But $c+d = e_1(0) \neq 0$,
and hence $e_0f_n \in \calE$ as claimed.

Finally, apply Lemma \ref{sectop.l.aux} 
above with $e = e_0 f_n$ to see that 
\[ \ran( \Phi(e_0f_n)) \subseteq \cls{\ran}(A).
\]
As
\begin{align*}
  \Phi(e_0f_n)D  & = D \Phi(e_0 f_n) = D \Phi(e_0)\Phi(f_n)=
\Phi(e_0)D\Phi(f_n) 
\\ & = (1 +A)^{-2}D\Phi(f_n) \to 
(1 + A)^{-2} T
\end{align*}
strongly, it follows that 
\beq\label{sectop.eq.aux2} 
\ran( (1 + A)^{-2} DT) \subseteq \cls{\ran}(A).
\eeq
Taking \eqref{sectop.eq.aux1} and \eqref{sectop.eq.aux2} 
together yields
\[   \ran( (1 + A)^{-2} DT) \subseteq \ker(A) \cap \cls{\ran}(A) 
= \{0\},
\]
since $A$ is sectorial \cite[Prop.{ }2.1]{HaaseFC}. This means
that $(1+A)^{-2}DT=0$, and since
since $D\in \calD$ was arbitrary,  it follows that $T=0$ as desired.
\end{proof}

As an immediate consequence we obtain
the result already announced above.

\begin{cor}\label{sectop.c.main}
The sectorial calculus $\Phi_A$ on 
$\Ha^\infty(\sector{\omega}\cup\{0\}) \cap \bdd(\Phi_A,\dom(\Phi_A))$
is closable with respect to the joint convergence structure
\[ \text{\rm ( bp-convergence on $\sector{\omega} \cup \{0\}$ ,
$\tau_A^s$-convergence
within $\calA_A$ )}.
\]
\end{cor}

Based on Corollary \ref{sectop.c.main}
we apply Theorem \ref{top.t.top-ext} 
and obtain  the {\emdf bp-extension}
$\Phi_A^\bp$ of the sectorial calculus $\Phi_A$ for $A$. 
Since the relevant function algebras are commutative,
there is no compatibility issue, cf. Corollary \ref{top.c.top-ext-comp}.

As bp-convergence is
weaker then uniform convergence, Corollary \ref{sectop.c.main}
implies in particular the result from \cite[Section 5]{Haa05c} 
that the sectorial calculus is closed with
respect to 
\[  \text{( uniform convergence on $\sector{\omega}\cup\{0\}$ , 
operator norm convergence ),} 
\]
which obviously is also a joint algebraic sequential convergence structure.
The respective topological extension 
(and also its canonical algebraic one) shall be called
the {\emdf uniform extension} of the sectorial calculus and 
denoted  by $\Phi_A^\uni$. 
Obviously, we have $\dom(\Phi_A^\uni) \subseteq \dom(\Phi^\bp_A)$ and 
$\Phi^\uni_A = \Phi^\bp_A$ on $\dom(\Phi_A^\uni)$.

\begin{rem}
The bp-extension is ``large'' in a sense, since bp-convergence
and $\tau_A^s$-convergence are relatively weak requirements.  
(Actually,  they are the weakest we can think of at the moment.) 
On the other hand, the uniform is quite ``small''. Whereas the
bp-extension is interesting in order to understand what a ``maximal'' 
calculus could be for a given operator, the uniform extension is
interesting in order to understand the ``minimal'' extension necessary
to cover a given function. 
\end{rem}


\section{Stieltjes Calculus and Hirsch Calculus}\label{s.hir}

The defect of the sectorial calculus  mentioned in the previous
section has, as a matter of fact, been observed 
by several other people working in the field. Of course, this has not
prevented people from working with operators
of the form \eqref{sectop.eq.Hirsch} or \eqref{sectop.eq.sgrp}. 
The former one has
actually been used  already by Hirsch in \cite{Hirsch1972}.
It was  extended algebraically by Martinez and Sanz in \cite{MartinezSanz1998,MartinezSanzTFPO}
under the name of ``Hirsch functional calculus''. 

Dungey in \cite{Dungey2009b} considers the operator $\psi(A)$
for
\[ \psi(z) = \int_0^1 z^\alpha \, \ud{\alpha} = \frac{z-1}{\log z}
\]
and remarks that $\psi(A)$ is defined within the Hirsch calculus
but not within the sectorial calculus.

Batty, Gomilko and Tomilov  in  \cite{BattyGomilkoTomilov2015} 
embed the Hirsch calculus (which is not
a calculus in our sense because the domain set is not an algebra)
into a larger calculus (in our sense) which is an algebraic 
extension of a calculus for the so-called  bounded Stieltjes algebra. 


\medskip

\subsection{The Stieltjes Calculus}\label{hir.s.sti}

In \cite{BattyGomilkoTomilov2015}, 
Batty, Gomilko and Tomilov define what they call
the {\em extended Stieltjes calculus} for a sectorial
operator. In this section, we introduce this calculus 
and show that it 
is contained in the uniform extension of the sectorial calculus.

According to \cite[Section 4]{BattyGomilkoTomilov2015}, 
 the {\emdf bounded Stieltjes
algebra} $\widetilde{\calS}_b$ consists of all functions
$f$ that have a representation
\beq\label{sectop.eq.stieltjes}
 f(z) = \int_{\R_+} \frac{\mu(\ud{s})}{(1 + sz)^m}
\eeq
for some $m\in \N_0$ and some $\mu \in \eM(\R_+)$. Each such $f$
is obviously holomorphic on $\sector{\upi}$ and bounded on each
smaller sector. It is less obvious, however, 
that $\widetilde{S}_b$ is a unital  algebra, a fact which 
is proved in \cite[Section 4.1]{BattyGomilkoTomilov2015}.

\begin{prop}\label{sectop.p.stieltjes}
Suppose $f$ is a bounded Stieltjes function with
representation \eqref{sectop.eq.stieltjes} and $A$ is a sectorial operator
on a Banach space $X$. 
Then 
\beq \label{sectop.aux.stieltjes2} 
\Phi^\uni_A(f) = \int_{\R_+} (1 + sA)^{-m}  \mu(\ud{s}),
\eeq
where $\Phi^\uni$ is the uniform extension of the sectorial
calculus for $A$.
\end{prop}

\begin{proof}
By subtracting $\mu\{0\}$ we may suppose that 
$\mu$ is supported on $(0, \infty)$. Also, we may
suppose that $m \ge 1$. 
Define
\[ f_n(z)  := \int_{[\frac{1}{n}, n]} \frac{\mu(\ud{s})}{(1+ sz)^m}
\quad
\text{and} \quad
a_n := \int_{[\frac{1}{n}, n]} \car \ud{\mu}.
\]
Then
\[ g_n(z) := f_n(z) - \frac{a_n}{1+z}
= \int_{[\frac{1}{n}, n]} \frac{1}{(1+ sz)^m} - \frac{1}{1 + z}
\, \mu(\ud{s}).
\]
The function under the integral is contained in $\calE$.
A moment's reflection reveals that $g_n \in \calE$ also, 
and that one can apply Fubini's theorem to compute $\Phi_A(g_n)$. 
This yields
\begin{align*}
 \Phi_A(g_n) & = \int_{[\frac{1}{n}, n]} 
\Phi_A\Bigl(\frac{1}{(1+ sz)^m} - \frac{1}{1 + z}\Bigr)
 \, \mu(\ud{s})
\\ & =  
\int_{[\frac{1}{n}, n]} 
(1+ sA)^{-m} - (1 + A)^{-1}  \, \mu(\ud{s})
\\ & = 
\int_{[\frac{1}{n}, n]} 
(1+ sA)^{-m} \, \mu(\ud{s}) - a_n (1 + A)^{-1},  
\end{align*}
and hence $f_n \in \calE_e$ with 
\[ \Phi_A(f_n)= \int_{[\frac{1}{n}, n]} 
(1+ sA)^{-m} \, \mu(\ud{s})
\to \int_{\R_+}
(1+ sA)^{-m} \, \mu(\ud{s})
\]
in operator norm. Since $f_n \to f$ uniformly on
$\sector{\omega}\cup\{0\}$ 
for
each $\omega < \upi$, the proof is complete.
\end{proof}

Instead of operating with a topological extension,  
the authors of \cite{BattyGomilkoTomilov2015} 
use \eqref{sectop.aux.stieltjes2} as a definition. 
They then have to show independence of the representation
\cite[Prop.{ }4.3]{BattyGomilkoTomilov2015}, 
compatibility with the holomorphic calculus
\cite[Lemma 4.4]{BattyGomilkoTomilov2015},  
and the algebra homomorphism property \cite[Propositions 4.5 and
4.6]{BattyGomilkoTomilov2015}. In our approach, all these facts follow
from   Proposition \ref{sectop.p.stieltjes}
and  general theory.\footnote{The reader might object that the
  ``general theory'' presented in this article is quite involved.
We agree, but stress the fact that only a commutative version of this
theory, which is relatively simple, is needed here.}

\medskip

\subsection{The Hirsch Calculus}\label{hir.s.hir}

Developing further the approach  of Hirsch \cite{Hirsch1972}, Martinez
and Sanz in \cite{MartinezSanz1998} and \cite{MartinezSanzTFPO} define
the class $\calT$ 
of all functions $f$ that have a representation
\[ f(z) = a + \int_{\R_+} \frac{z}{1 + zt} \, \nu(\ud{t}),
\]
where $a\in \C$ and $\nu$ is a Radon measure on $\R_+$
satisfying
\[ \int_{\R_+} \frac{ \abs{\nu}(\ud{t})}{1+t} < \infty.
\]
One can easily see that $\calT$ is contained in the algebraically extended
Stieltjes algebra. Just write
\[ f(z) = a + z \int_{[0,1]} \frac{1}{1 + tz} \, \nu(\ud{t})
+ \int_{(1, \infty)} \frac{z}{1 + zt} \, \nu(\ud{t})
:= a + z g(z) + h(z)
\] 
and note that $g$ and $h$ are bounded Stieltjes functions. 
The latter  is obvious for $g$; and for $h$ it follows from the
identity
\beq\label{sectop.eq.Hirsch-aux}
    h(z) = \int_{(1, \infty)} \frac{z}{1 + zt} \, \nu(\ud{t})
=  \int_{(1, \infty)} \frac{\nu(\ud{t})}{t} - 
\int_{(1, \infty)} \frac{1}{1 +  tz}\, \frac{\nu(\ud{t})}{t}.
\eeq
We obtain
\begin{align*}  \Phi_A^\uni(f) & = a + A \Phi_A^\uni(g) + \Phi_A^\uni(h)
\\ & = a + A \int_{[0,1]} (1 + tA)^{-1} \, \nu(\ud{t})
+ \int_{(1, \infty)} A (1+tA)^{-1} \, \nu(\ud{t})
\end{align*}
by a short computation using \eqref{sectop.eq.Hirsch-aux}. 
This  coincides
with how $f(A)$ is defined in  \cite[Def.{ }4.2.1]{MartinezSanzTFPO}. 
Hence, the Hirsch calculus
(which is not a functional calculus in our terms since $\calT$ is not
an algebra) is 
contained in the uniform extension of the sectorial calculus.

\medskip

\subsection{Integrals  involving  Holomorphic Semigroups}\label{hir.s.hol}

If $A$ is sectorial of angle $\omega_{\sct}(A) < \frac{\upi}{2}$, then 
$-A$ generates a holomorphic semigroup
\[ T_A(\lambda ) := \ue^{-\lambda A} := (\ue^{-\lambda \bfz})(A)
\qquad (\lambda \in \sector{\omega_\sct(A) - \frac{\upi}{2}}),
\]
see \cite[Section 3.4]{HaaseFC}. If one has $\alpha = 0$ or $\re
\alpha > 0$, and one restricts $\lambda$ to a
smaller
sector, the function $\lambda \mapsto (\lambda A)^\alpha T_A(\lambda)$
becomes  uniformly bounded. Hence, one can integrate with respect to a
bounded measure.   The following result shows that also these
operators are covered by the uniform extension of the sectorial
calculus.

\begin{prop}\label{hir.p.hol}
Let $0 \le \vphi <  \frac{\upi}{2}$, let $\mu$ be a complex Borel measure
on $\cls{\sector{\vphi}}$ and let $\alpha = 0$ or $\re \alpha > 0$.
 Then the function
\[ f(z) := \int_{\cls{\sector{\vphi}}} (\lambda z)^\alpha 
\ue^{-\lambda z}\, \mu(\ud{\lambda})
\]
is holomorphic on $\sector{\frac{\upi}{2} - \vphi}$ and uniformly
bounded on each smaller sector. If $A$ is any sectorial operator
on a Banach space with $\omega_{\sct}(A) + \vphi < \frac{\upi}{2}$, 
then  
\[ \Phi_A^\uni(f) =  \int_{\cls{\sector{\vphi}}} (\lambda A)^\alpha 
\ue^{-\lambda A}\, \mu(\ud{\lambda}),
\] 
where $\Phi^\uni$ is the uniform extension of the sectorial
calculus for $A$. 
\end{prop}

We point out that Proposition \ref{hir.p.hol} applies in particular
to the case that  $\vphi = 0$ and $\cls{\sector{\vphi}} = \R_+$ is just the
real axis.

\begin{proof}
The proof follows the line of the proof of Proposition
\ref{sectop.p.stieltjes} 
and
we only sketch it. First one subtracts a constant to reduce to the
case that $\mu$ has no mass at  $\{0\}$.
Then one uses the approximation
\[ \int_{\lambda \in \cls{\sector{\vphi}}, \frac{1}{n}\le
  \abs{\lambda} \le n} \dots\, \mu(\ud{\lambda})
\quad\to\quad   \int_{\cls{\sector{\vphi}}} \dots\, \mu(\ud{\lambda}) 
\qquad (n \to \infty)
\]
first for scalars and then for operators. This reduces the
claim to establishing the identity
\[ \Phi_A\Bigl(\int_{\lambda \in \cls{\sector{\vphi}}, \frac{1}{n}\le
  \abs{\lambda} \le n} (\lambda \bfz)^\alpha \ue^{-\lambda
  \bfz}\mu(\ud{\lambda})\Bigr)
 = 
\int_{\lambda \in \cls{\sector{\vphi}}, \frac{1}{n}\le
  \abs{\lambda} \le n} (\lambda A)^\alpha \ue^{-\lambda A}\,
\mu(\ud{\lambda}).
\]
If $\re \alpha > 0$ then this is a simple application of Fubini's
theorem. If $\alpha = 0$ then one has to write
\[ \ue^{-\lambda \bfz} =  \frac{1}{1 + \bfz} + \Bigl( \ue^{-\lambda
  \bfz}
- \frac{1}{1+\bfz}\Bigr)
\]
and use Fubini for the second summand.
\end{proof}

\section{Semigroup and Group Generators}\label{s.sgr}

We define a {\emdf bounded semigroup}
to be uniformly bounded mapping  $T: \R_+ \to \BL(X)$ 
which is  strongly continuous on $(0, \infty)$ and satisfies
 the semigroup laws
\[ T(0)= \Id, \qquad  T(s+t) = T(s) T(t) \qquad (t,s > 0).
\]
(This has been called a {\em degenerate semigroup} in
\cite{HaaseFC}.) For $\mu \in \eM(\R_+)$ one can define
\[ \Psi_T(\mu) := \int_{\R_+} T(s) \, \mu(\ud{s}) \in \BL(X)
\]
as a strong integral. The mapping
\[ \Psi_T : \eM(\R_+) \to \BL(X)
\]
is an algebra homomorphism with respect to the convolution product. 
There is a unique linear relation $B$ on $X$, called the
{\emdf generator} of $T$,  such that 
\[    (\lambda - B)^{-1} =  \int_0^\infty \ue^{-\lambda t}T(t)\,
\ud{t}
\]
for one/all $\lambda \in \C$ with $\re \lambda > 0$ 
\cite[Appendix A.8]{HaaseFC}. 

Because of $\Psi_T(\delta_0) = \Id$, the representation $\Psi_T$ is not
degenerate.  However, its restriction to $\eM(0, \infty)$ might be. 
In fact, this is the case if and only  if the common kernel
$\bigcap_{t> 0} \ker(T(t))$ is not trivial, if and only
if $B$ is not operator.  From now one, 
we confine ourselves to the non-degenerate case, i.e., 
we suppose that $B$ is an operator. Instead of at $B$ we shall be
looking at
\[ A := -B
\]
it the following.

\medskip
Note that the Laplace transform 
\[  \Lap: \eM(\R_+) \to \Cb(\cls{\C_+}) \qquad \Lap \mu(z) :=
\int_{\R_+} \ue^{-zs}\,\mu(\ud{s}) \qquad (\re z \ge 0).
\]
is injective. Here, $\C_+ := \{ z\in \C \suchthat \re z > 0\} = 
\sector{\frac{\upi}{2}}$. Its image is the {\emdf Hille--Phillips algebra}
\[  \calL\calM(\C_+) := \{ \Lap \mu \suchthat \mu \in \eM(\R_+)\},
\]
a unital algebra under pointwise multiplication. The mapping
\[ \Phi_T : \calL\calM(\C_+) \to \BL(X), \qquad \Phi_T(f) =
\Psi_T(\Lap^{-1} f)
\]
is called the {\emdf Hille--Phillips calculus} (HP-calculus, for
short)
for $A$. One has
\[   \Phi_T((\lambda + \bfz)^{-1}) =  (\lambda +A)^{-1}
\]
for all $\lambda \in \C_+$.  Finally, we  extend $\Phi_T$ algebraically within the field
$\Mer(\C_+)$ of meromorphic functions on $\C_+$ and call this the
{\emdf extended Hille--Phillips calculus}.

\medskip

\subsection{The Complex Inversion Formula}\label{sgr.s.coi}

The semigroup can be reconstructed from its generator by the so-called
complex inversion formula. This is a standard fact from 
semigroup theory in the case that $T(t)$ is strongly continuous at $t=0$, i.e., 
if $A$ is densely defined. However, we
do not want to make this assumption here, so we need to
digress a little on that topic.

\begin{prop}[Complex inversion formula]\label{sgr.p.coi}
Let $-A$ be the generator of a bounded semigroup $T= (T(t))_{t> 0}$. 
Then the mapping
\[  \R_+ \to \BL(X)\qquad  t \mapsto T(t)(1 +A)^{-1}
\]
is Lipschitz-continuous in operator norm. Moreover, for each $\omega <
0$ 
\[ T(t)(1+A)^{-2} = \frac{1}{2\upi \ui} \int_{\omega + \ui \R} 
\frac{\ue^{-tz}}{(1+z)^2} \, R(z, A)\, \ud{z} \qquad (t\ge 0)
\]
where the integration contour is directed top down from   
from $\omega +\ui\infty$  to $\omega - \ui \infty$. 
\end{prop}

\begin{proof}
The HP-calculus turns the scalar identity
\[  \ue^{-tz\bfz} - \ue^{-s\bfz} = -\bfz \int_s^t \ue^{-r\bfz}\,
\ud{r}
\qquad (s,t\in \R_+)
\]
 into the operator identity
\[  T(t) - T(s) = -A \int_s^t T(r)\, \ud{r}.
\]
Multiplying with $(1+A)^{-1}$ from the right and estimating
yields
\[ 
 \norm{T(t)(1+A)^{-1} - T(s)(1+A)^{-1}}
 \le  M \norm{A(1 +A)^{-1}} \abs{s-t} \qquad (s,t\in \R_+).
\]
For the second claim we note first the estimate
\beq\label{sgr.eq.sthp}   \norm{(\lambda + A)^{-1}} \le \frac{M}{\re \lambda}\qquad 
(\re \lambda > 0),
\eeq
where $M := \sup_{t > 0} \norm{T(t)}$. Consequently,
$A$ is an operator of 
strong right half-plane type  $0$, and hence admits a functional
calculus $\Psi$, say, on half planes as in \cite{BatHaaMub2013}. 
Writing $\ue^{-tA} := \Psi(\ue^{t\bfz})$ one obtains
\[ S(t) := \frac{1}{2\upi \ui} \int_{\omega + \ui \R} 
\frac{\ue^{-tz}}{(1+z)^2} \, R(z, A)\, \ud{z} 
= \Psi( \frac{\ue^{-t\bfz}}{(1+\bfz)^2} ) = \ue^{-tA}(1+A)^{-2}
\]
for $t\ge 0$ by definition of $\Psi$ and usual functional calculus rules. Taking Laplace
transforms,
by \cite[Lemma 2.4]{BatHaaMub2013} we obtain
\beq\label{sgr.eq.coi-aux}
  \int_0^\infty \ue^{-\lambda t} S(t)\, \ud{t}
=   \int_0^\infty \ue^{-\lambda t} \ue^{-tA}(1+A)^{-2}\, \ud{t}
= (\lambda + A)^{-1}(1+A)^{-2}
\eeq
whenever $\re \lambda > -\omega$. Since $\omega$ can be chosen arbitrarily
close to $0$, the identity \eqref{sgr.eq.coi-aux} 
actually holds for all $\re \lambda > 0$. 
Since the Laplace transform is injective, it follows that 
\[    S(t) = T(t)(1+A)^{-2} \qquad (t\ge 0)
\]
as claimed.
\end{proof}

As a consequence we obtain that the commutant of the semigroup and
the commutant of its generator coincide.

\begin{cor}\label{sgr.c.commutator}
For a bounded operator $S\in \BL(X)$ the following assertions are
equivalent:
\begin{aufzii}
\item $S$ commutes with $(1+A)^{-1}$
\item $S$ commutes with each $T(t)$, $t> 0$. 
\end{aufzii}
\end{cor}

\begin{proof}
The implication (ii)$\dann$(i) is trivial. Suppose that (i) holds.
Then $S$ commutes with $R(\lambda,A)$ for each $\lambda \in
\resolv(A)$ \cite[Prop.{ }A.2.6]{HaaseFC}. 
By the complex inversion formula, $S$ commutes
with $T(t) (1+A)^{-2} = (1+A)^{-2}T(t)$. It follows that 
\[   (1+A)^{-2} ST(t) = S (1+A)^{-2}T(t) = 
(1+A)^{-2} T(t) S.
\]
Since $(1+A)^{-2}$ is injective, $T(t)S = ST(t)$.
\end{proof}

\medskip

\subsection{A Topological Extension of the HP-Calculus}\label{sgr.s.topext}

We let, as before, $-A$ be the generator of a bounded semigroup as
above.  As in Section \ref{s.sectop} we consider the algebra
\[ \calA_A = \{ (1+A)^{-1}\}'
\]
which by Corollary \ref{sgr.c.commutator} 
coincides with the commutant of the semigroup.
Similarly to Section \ref{s.sectop}, for $((T_n)_n, T) \in \calA_A^\N\times
\calA_A$
we write
\[ T_n \stackrel{\tau_A^n}{\to} T
\]
if there is a point-separating subset $\calD \subseteq \calA_A'$ such
that 
\[    DT_n \to DT \quad \text{in operator norm, for each $D\in \calD$} 
\]
Note the difference to the structure $\tau_A^s$ considered in 
Section \ref{s.sectop}, where we allowed strong 
convergence. It is easily checked that $\tau_A^n$ is an algebraic
Hausdorff  convergence structure.

\begin{thm}\label{sgr.t.topext}
Let $-A$ be the generator of a bounded semigroup $T$ on a Banach space
$X$. Then the  Hille--Phillips calculus $\Phi_T$ is closable 
with respect to the joint convergence structure
\[  \text{\rm (pointwise convergence on $\cls{\C_+}$  ,
  $\tau_A^n$-convergence ) }
\quad \text{on}\quad   (\Ha^\infty(\cls{\C_+}) , \calA_A).
\]
 \end{thm}

\begin{proof}
Suppose that $f_n = \Lap\mu_n \in \calL\calM(\C_+)$ is pointwise convergent
on $\cls{\C_+}$ to $0$ and that $\calD \subseteq \calA_A'$ is a
point-separating subset of $\calA_A$ such that  
\[ D\Phi_T(f_n) \to DT
\]
in operator norm. It suffices to show that $(1+A)^{-1}DT=0$. 

To this aim, note that $\calA_A'$ is a commutative unital Banach
algebra. Since $\Phi_T(f_n) \in \calA_A'$, we have also $DT \in
\calA_A'$.
Hence, by Gelfand theory, it suffices to show that 
\[ \chi( (1+A)^{-1}DT) = 0
\]
for each multiplicative linear functional $\chi: \calA_A'\to
\C$. Fix such a functional $\chi$. Since
\[ \chi( (1+A)^{-1}DT)  = \chi( (1+A)^{-1}) \chi(DT)  
\]
we may suppose without loss of generality
that $\alpha := \chi((1+A)^{-1}) \neq 0$.

Consider the function $c: \R_+ \to \C$, $c(t) := \chi( T(t))$. Then
\[  c(t+s) = c(t) c(s)\qquad (t,s \ge 0)
\]
and $c$ is bounded. Moreover, 
\[ t \mapsto \alpha c(t) = \chi( T(t)(1+A)^{-1})
\]
is continuous (by Proposition \ref{sgr.p.coi}).  Since $\alpha \neq 0$, 
$c$ is continuous. It follows from a classical theorem of Cauchy that
there is $\lambda \in \cls{\C_+}$ such that 
\[ c(t) = \ue^{-\lambda t} \qquad (t\ge 0).
\]
Next, we find  
\[ \Phi(f_n)(1+A)^{-1} = \int_{\R_+} T(t)(1+A)^{-1}\, \mu_n(\ud{t}).
\]
By Proposition \ref{sgr.p.coi}, the integrand is a norm-continuous function
of $t$. Hence,
\begin{align*}
 \chi( D\Phi_T(f_n)(1+A)^{-1}) & = 
\chi(D) \int_{\R_+} \chi( T(t)(1+A)^{-1}) \, \mu_n(\ud{t})
\\ & = \alpha \chi(D) \int_{\R_+} \ue^{-\lambda t} \, \mu_n(t)
= \alpha \chi(D) f_n(\lambda).
\end{align*}
Letting $n \to \infty$ yields
\[  \chi((1+A)^{-1}DT) = 0
\]
as desired. (It is in this last step that we need the operator norm
convergence $D\Phi_T(f_n)\to DT$.)
\end{proof}

\begin{rem}
It is not difficult to see that the spectra  of $(1+A)^{-1}$  in
$\calA_A'$ and in $\BL(X)$ coincide. It follows from Gelfand theory that 
\[ \spec((1+A)^{-1}) =\{ \chi((1+A)^{-1}) \suchthat  0 \neq \chi 
\,\, \text{is a multiplicative functional on $\calA_A'$}\}.
\]
This implies, eventually, that if $\alpha := \chi( (1+A)^{-1}) \neq 0$ then 
$\chi(T(t)) = \ue^{-\lambda t}$, where
$(1 + \lambda)^{-1} = \alpha \in \spec( (1+A)^{-1})$, and
hence $\lambda \in \spec(A)$ by the spectral mapping theorem for the
resolvent. All in all we obtain that we can replace pointwise
convergence on $\cls{\C_+}$ by pointwise convergence on $\spec(A)$ in 
Theorem \ref{sgr.t.topext}.  (These arguments actually show that 
the failing of the spectral mapping theorem for the semigroup is
precisely due to the existence of multiplicative functionals $\chi$ on
$\calA_A'$ that  vanish on $(1+A)^{-1}$ but
do not vanish on some $T(t)$. However, these functionals are irrelevant in
our context.) 
\end{rem}

According to Theorem \ref{sgr.t.topext}, the Hille--Phillips calculus $\Phi_T$
has a topological extension based on the joint convergence structure
\[  \text{\rm (pointwise convergence on $\cls{\C_+}$  ,
  $\tau_A^n$-convergence ) }
\quad \text{on}\quad   (\Ha^\infty(\cls{\C_+}) , \calA_A).
\]
Let us call this the {\emdf semi-uniform extension} of the HP-calculus.
We do not know whether one can replace $\tau_A^n$ by $\tau_A^s$ here
in general. However, there are special cases, when it is possible.

\medskip

\subsection{Compatibility of the HP-Calculus and the Sectorial Calculus}
\label{sgr.s.sec-sgr}

By \eqref{sgr.eq.sthp}, the negative generator $A$ of the bounded
semigroup $T$  is sectorial of angle $\omega_{\sct}(A) \le
\frac{\upi}{2}$. 
Hence, there are now two competing functional calculi for it, the
Hille--Phillips calculus $\Phi_T$ and the sectorial calculus $\Phi_A$, 
each coming with its associated algebraic and topological extensions.
Of course, we expect  compatibility, so let us have a closer look.

\medskip
Suppose first that $\omega_{\sct}(A) < \frac{\upi}{2}$. Then 
by Proposition \ref{hir.p.hol}, the Hille--Phillips calculus
is a restriction of the uniform extension of the elementary sectorial
calculus for $A$. By commutativity, compatibility is still valid for the
respective algebraic extensions, that is: the extended Hille--Phillips
calculus is a subcalculus of $\Phi_A^\uni$.

Now, suppose that $\omega_{\sct}(A) = \frac{\upi}{2}$. 
It has been shown in \cite[Lemma 3.3.1]{HaaseFC} that each $e \in
\calE[\sector{\frac{\upi}{2}}]$ is contained in $\calL\calM(\C_+)$ with 
\[ \Phi_T(e) = \Phi_A(e)
\]
(The actual formulation of \cite[Lemma 3.3.1]{HaaseFC} yields a little less, but
its proof works in the more general situation considered here.) 
It follows that 
\[  \calE_e[\sector{\frac{\upi}{2}}] \subseteq \calL\calM(\C_+)\quad
\text{and}\quad \Phi_T = \Phi_A\,\,
\text{on}\,\, \calE_e[\sector{\frac{\upi}{2}}].
\]
By commutativity of the algebras, the algebraic extensions of these
calculi also are compatible (Theorem \ref{ext.t.succ-comp}). That is, the
(algebraically) extended
Hille--Phillips calculus is an extension of the sectorial calculus for
$A$. Furthermore, the uniform extension of the sectorial calculus
is clearly contained in the semi-uniform extension of the
Hille--Phillips calculus as described above.

\begin{rems}
\begin{aufziii}
\item  The bounded Stieltjes algebra is actually included in the
Hille--Phillips
algebra. This can be seen by a direct computation. More generally,
each function $f\in \Mer[\sector{\frac{\upi}{2}}]$ such that 
$\Phi_A(f)$ is bounded for {\em each} negative generator of a bounded
semigroup, is contained in $\calL\calM(\C_+)$. (Choose $T$ to be the 
right  semigroup on $\Ell{1}(\R_+)$.) 

\item  At present, we do not know how the bp-extension of the sectorial
calculus and the semi-uniform extension of the HP-calculus relate. 

\end{aufziii}
\end{rems}

\medskip
On the other hand, we can ``reach'' the HP-calculus 
from the sectorial calculus by employing a modification of the
uniform extension. Namely, consider the joint convergence structure
\beq\label{sgr.eq.joint-uni} 
\text{\rm ( uniform convergence on $\C_+$ , operator norm convergence  ) 
}
\eeq
on $\calL\calM(\C_+)$ and $\calA_A$, respectively. 
By compatibility and Theorem \ref{sgr.t.topext}, the sectorial calculus
on $\calE_e[\sector{\frac{\upi}{2}}]$ is closable with respect to that
structure. The next result shows that the functions
\[  \frac{\ue^{-t\bfz}}{(1+ \bfz)^2} \qquad (t >  0)
\]
are in the domain of the corresponding topological extension.

\begin{lem}\label{sgr.l.coi-approx}
Let $-A$ be the generator of a bounded semigroup $T$, and let 
$t > 0$ and $\omega < 0$.
For any  $n\in \N$ the function
\[ f_n(z) := \frac{1}{2\upi \ui} \int_{\omega + \ui[-n,n]} 
\frac{\ue^{-wt}}{(1+w)^2} \frac{\ud{w}}{w-z}
\]
is contained in $\calE_e[\sector{\frac{\upi}{2}}]$. Moreover, 
\[ f_n \to \frac{\ue^{-t\bfz}}{(1+ \bfz)^2}\quad (n \to \infty) 
\]
uniformly on $\cls{\C_+}$ and 
\[  \Phi_A(f_n) \to   T(t)(1+A)^{-2}
\]
in operator norm. 
\end{lem}

\begin{proof}
Note that $f_n$ is holomorphic on 
$\C \ohne (\omega {+}\ui[-n,n])$ and hence on a sector
$\sector{\vphi}$ for $\vphi > \upi/2$. On each smaller sector
we have $f_n(z) = O(\abs{z}^{-1})$ as $\abs{z} \to \infty$ and 
$f_n(z) -  f_n(0) = O(\abs{z})$ as $\abs{z} \to 0$. It follows
that $f_n \in \calE_e(\sector{\vphi})$.

The remaining statements follow from the complex inversion formula.
One needs the identity
\[ \Phi_A(f_n) = \frac{1}{2\upi \ui} \int_{\omega + \ui[-n,n]} 
\frac{\ue^{-wt}}{(1+w)^2}\, R(w,A)\ud{w}, 
\]
which is proved by standard arguments.
\end{proof}

Since $T(t)$ can be reconstructed algebraically from $T(t)
(1+A)^{-2}$, we see that the semigroup operators are contained
in the algebraic extension of the topological extension given by
\eqref{sgr.eq.joint-uni} of the elementary sectorial calculus.

\section{Normal Operators}\label{s.spt}

Normal operators on Hilbert spaces are known, by the spectral theorem,
to have the best functional calculus one can hope for. The
(Borel) functional calculus for a normal operator is heavily
used in many areas of mathematics and mathematical
physics. Despite this importance of the functional calculus, 
the spectral theorem is most frequently
formulated in terms of projection-valued measures or multiplication operators,
and the functional calculus itself appears merely as a derived concept.  

This expositional dependence 
(of the functional calculus on the spectral measure)
is manifest in  the classical extension 
of the calculus from bounded to unbounded
functions as described, e.g.,  in Rudin's book \cite{RudinFA}. 
Since the description of $f(A)$ for unbounded $f$ in terms of spectral
measures is far from simple, working with the unbounded part
of the calculus on the basis of this exposition is rather cumbersome.

However, the situation now is different from when Rudin's classic text
was written, in at least two respects. Firstly, we now have an axiomatic notion
of a functional calculus (beyond bounded operators in its
range). This enables us to develop the properties of the calculus
from axioms rather than from a particular construction, which makes
things far more perspicuous and, eventually, far easier to handle.

\vanish{
In the classical approach, assertions like (FC2)
appear as consequences of the construction (cf.{ }\cite[13.24]{RudinFA}). 
Consequences of (FC2)  are itself not acknowledged as such, and hence remain 
dependent on the construction as well. In contrast, we can now develop
the theory of Borel functional calculus from a minimalistic axiom
system and integrate the construction steps of such a functional calculus
into the theory as separate theorems.
}

Secondly, we now have an elegant tool to go from bounded to unbounded
functions: the algebraic extension procedure. As a result, the
unbounded part of the construction of the functional calculus for a
normal operator on a Hilbert space just becomes a corollary of
Theorem \ref{ext.t.ext}. Actually,  all algebras in this context are
commutative and there is always an anchor element, so one does not even need
the full force of Theorem \ref{ext.t.ext}, but only the relatively elementary methods
of \cite{HaaseFC}. 

In order to render these remarks less cryptic, we need of course 
be more specific. We shall sketch the main features below.  A
more detailed treatment can be found in the separate paper
\cite{Haase2020bpre}.

\medskip

Let $(X,\Sigma)$ be a {\em measurable space},
i.e.,  $X$ is a set and  $\Sigma$ is a  $\sigma$-algebra of subsets of
$X$. We let
\begin{align*} \Meas(X,\Sigma) & := \{ f: X \to \C \suchthat \text{$f$ measurable}\}.
\end{align*}
A {\emdf measurable (functional) calculus}  
on $(X,\Sigma)$ is a pair
$(\Phi, H)$ where $H$ is a Hilbert space and 
\[ \Phi: \Meas(X,\Sigma)\to \Clo(H)
\]
is a mapping with the following properties ($f,\: g \in \Meas(X,\Sigma),\: \lambda \in \C$):
\begin{aufziii}
\item[\quad(MFC1)] $\Phi(\car) = \Id$;
\item[\quad(MFC2)] $\Phi(f) + \Phi(g) \subseteq \Phi(f+g)$ and 
$\lambda \Phi(f) \subseteq \Phi(\lambda f)$;
\item[\quad (MFC3)] $\Phi(f)\Phi(g) \subseteq \Phi(fg)$\quad  and 
\[ \dom(\Phi(f)\Phi(g)) = \dom(\Phi(g))\cap \dom(\Phi(fg));
\]
\item[\quad(MFC4)] $\Phi(f) \in \BL(H)$  and $\Phi(f)^* =
  \Phi(\konj{f})$  if $f$ is bounded;
\item[\quad(MFC5)]  If $f_n \to f$ pointwise and boundedly, then
  $\Phi(f_n) \to \Phi(f)$ weakly.  
\end{aufziii}
Property (MFC5) is called the {\emdf weak bp-continuity} of the mapping
$\Phi$. 

\medskip
Evidently, (MFC1)--(MFC3) are just the axioms (FC1)--(FC3) of a
proto-calculus. For  $f\in \Meas(X, \Sigma)$ let
\[ e := \frac{1}{1 + \abs{f}}
\]
Then $e$ is a bounded function and $ef$ is also bounded. 
Hence, by (MFC4), $e$ is a regularizer of $f$. 
Moreover, $\Phi(e^{-1})$ is defined, and hence $\Phi(e^{-1}) = \Phi(e)^{-1}$
(Theorem \ref{afc.t.pro-cal}). It follows that
\[ \Phi(f) = \Phi(e^{-1} e f) = \Phi(e)^{-1} \Phi(ef),
\]
which just means that the set $\{ e\}$ is determining for $\Phi(f)$. 
This show that 
\[ \calE := \{ e\in \Meas(X, \Sigma) \suchthat \text{$e$ is
  bounded}\}
\]
is an algebraic core for $\Phi$. 
(In particular, $\Phi$ satisfies (FC4) and hence is a calculus.) 

As a result, each measurable calculus
coincides with the algebraic extension of its restriction
to the bounded functions. To construct a measurable calculus,
it therefore suffices to construct a calculus on the bounded measurable
functions and then apply the algebraic extension procedure. And this
is a far simpler method than employing spectral measures.

\medskip

It is remarkable (and very practical) that only (MFC1)--(MC5) are
needed to establish all the well-known properties of the Borel
calculus for normal operators. For example, one can prove
that the identity
\[ \Phi(\konj{f}) = \Phi(f)^*
\]
holds for each $f\in \Meas(X, \Sigma)$, and not just for bounded
functions as guaranteed by (MFC4). Next, observe that 
for given $f,g$ the 
sequence of functions
\[ e_n := \frac{n}{n+\abs{f} + \abs{g}}
\]
form a common  approximate identity for $f$ and $g$. (This is actually
a strong approximate identity, since strong convergence in (MFC5)  
holds automatically.) By Theorem \ref{api.t.api} we obtain  
\[ \cls{\Phi(f) + \Phi(g)} = \Phi(f+g),\qquad \cls{\Phi(f)\Phi(g)}
= \Phi(fg).
\]
One of the most important results in this abstract development
of measurable calculi concerns uniqueness. We only cite a 
corollary of a more general theorem:

\begin{thm}
Let $X\subseteq \C^d$, endowed with the trace $\sigma$-algebra
of the Borel algebra.
Let $(\Phi, H)$ and $(\Psi,H)$ be two measurable calculi on $X$
such that
\[ \Phi(\bfz_j) = \Psi(\bfz_j) \quad (j=1, \dots, d).
\]
Then $\Phi = \Psi$. 
\end{thm}

This theorem implies, e.g., the composition rule 
\[ (f\nach g)(A) = f( g(A))
\]
for a normal operator $A$ on a Hilbert space $H$, since
both mappings 
\[ \Phi(f) := (f\nach g)(A) \quad \text{and}\quad \Psi(f) := f(g(A))
\]
are Borel calculi on $\C$ that agree for $f = \bfz$. 

\medskip

For more about the functional calculus approach to the spectral
theorem we refer to \cite{Haase2020bpre}.

\medskip


\appendix

\section{The  Closed Graph Theorem for the Weak$ ^*$-Topology}\label{app.cgt}

In our investigations on the dual calculus,
the following theorem   is needed.

\begin{thm}\label{sal.t.cgt}
Let $X,Y$ be  Banach spaces and $T \in \BL(X',Y')$ such that $T$ has a
closed graph with respect to the weak$^*$ topologies. Then $T$ is continuous
with respect to the weak$^*$ topologies and hence of the form
$T= S'$ for some $S\in \BL(Y;X)$. 
\end{thm}

Theorem \ref{sal.t.cgt} is actually  a special case of much
more general results about operators between certain topological
vector spaces. One of the earliest references for it is \cite[Theorem
1]{McIntosh1969}. A more explicit version for Fr\'echet spaces is
 \cite[p.79]{KoetheTVS2}, with the caveat that one really needs 
a closed graph (rather than just a sequentially closed graph) to make
the proof work. Again less explicit, Theorem \ref{sal.t.cgt}
is a special case of the results in  \cite{Ruess1977}.
(I am indebted to Wolfgang Ruess for these bibliographical remarks.)

\medskip
However, as all these references rely on expert knowledge
in the field of topological vector spaces, we include
a short ad hoc proof for the convenience of the reader.

\begin{proof}[Proof of Theorem \ref{sal.t.cgt}]
Without loss of generality, we may suppose that $\norm{T}\le 1$. 
Let $K := \Ball_{X'}[0,1]$ and $L := \Ball_{Y'}[0,1]$ be the closed
unit balls of $X'$
and $Y'$, respectively. Then $K, L$ are compact with respect to the
weak$^*$ topologies. Clearly, $T(K) \subseteq L$, and 
by hypothesis, $T\res{K} : K \to L$ has a closed graph. Hence,
$T\res{K}$ is continuous \cite[\S26, Ex.8]{MunkresTop}.
 This implies that for each $y\in Y$ the element
$T'y$ of $X''$ is weak$^*$ continuous on $K$, so is an element of $X$
by a classical result (see
\cite[1.2]{StratilaZsidoLvNA} for a simple proof).
This means that 
$T'$ maps $Y$ into $X$ and with  $S := T\res{Y}$ the claim follows.
\end{proof}


\def\cprime{$'$} \def\cprime{$'$} \def\cprime{$'$} \def\cprime{$'$}

\medskip
\subsection*{Acknowledgements}

In preliminary form, parts of this work have appeared in 
 the lecture notes to the
21st International Internet Seminar on ``Functional Calculus'' 
during the academic year 2017/2018. I am  indebted to the participating
students  and  colleagues  for valuable remarks and
discussions. 

I also want to thank my colleague and friend Yuri Tomilov for
showing interest in this work and for some helpful comments.

This work was completed while I was spending a research
sabbatical at UNSW in Sydney. I am  grateful to
Fedor Sukochev for his kind invitation. Moreover, 
I  gratefully acknowledge the financial support from the DFG,
 project number 431663331.
\end{document}